\documentclass{tran-l}

\usepackage{amssymb}

\usepackage{srcltx}
\usepackage{latexsym,amsfonts,amssymb,amsmath,longtable,amsthm,
mathrsfs,cite
}
\usepackage{multirow}

\DeclareMathAlphabet{\mathpzc}{OT1}{pzc}{m}{it}
\renewcommand{\leq}{\leqslant}
\renewcommand{\geq}{\geqslant}

\newcommand{\la}{\langle}
\newcommand{\ra}{\rangle}
\renewcommand{\L}{\hat{L}}
\newcommand{\Oo}{
\mathcal
O}

\newcommand{\BS}{{
\mathcal B}{\mathcal S}}

\newtheorem{theorem}{Theorem}
\newtheorem{lemma}{Lemma}[section]
\newtheorem{Conj}{Conjecture}

\newtheorem*{BSThm}{Baer-Suzuki Theorem}
\newtheorem{Prop}{Proposition}

\theoremstyle{definition}
\newtheorem{definition}[theorem]{Definition}

\theoremstyle{remark}

\numberwithin{equation}{section}

\DeclareMathOperator{\Inn}{Inn}  \DeclareMathOperator{\Inndiag}{Inndiag}
 \DeclareMathOperator{\GL}{GL}
\DeclareMathOperator{\Aut}{Aut} 
\DeclareMathOperator{\Sym}{Sym}

\DeclareMathOperator{\bs}{BS}

\renewcommand{\qed}{$\blacksquare$}

\begin{document}

\title{Baer--Suzuki theorem for the $\pi$-radical}


\author[Yang, N.]{Nanying Yang}
\address{Jiangnan University\\
no. 1800, lihu avenue, wuxi, jiangsu, 214122}
\curraddr{}
\email{yangny@jiangnan.edu.cn }
\thanks{}

\author[Revin D.O.]{Danila O. Revin}
\address{Sobolev Institute of Mathematics,\\
pr-t Acad. Koptyug, 4,\\
630090, Novosibirsk, Russia}
\curraddr{}
\email{revin@math.nsc.ru}
\thanks{}

\author[Vdovin E.P.]{Evgeny P. Vdovin}
\address{Mathematical Center in Akademgorodok,\\
Sobolev Institute of Mathematics,\\
pr-t Acad. Koptyug, 4,\\
630090, Novosibirsk, Russia}
\curraddr{}
\email{vdovin@math.nsc.ru}
\thanks{}


\date{}

\dedicatory{}

\commby{}

\begin{abstract}
In the paper we prove (modulo the classification of finite simple groups) an analogue of the famous Baer-Suzuki theorem for the $\pi$-radical of a finite group, where $\pi$ is a set of primes.
\end{abstract}

\maketitle


\section*{Introduction}


Throughout the paper we denote by  $\pi$ a set of primes. A finite group is called a  {\em $\pi$-group},
if all prime divisors of its order belong to~$\pi$. Given a finite group $G$, by  $\Oo_\pi(G)$ we denote its {\em $\pi$-radical}, i.~e. the largest normal $\pi$-subgroup of~$G$,
and by $G^\sharp$ we always denote $G\setminus \{1\}$.

The Baer--Suzuki theorem  \cite{Baer,Suz,AlpLy} states

\begin{BSThm}
Let $p$ be a prime,  $G$ a finite group, and $x\in G$.

Then $x\in \Oo_p(G)$ if and only if  $\la x,x^g \ra$ is a  $p$-group for every~${g\in G}$.
\end{BSThm}

Clearly, in this theorem only the ``if'' part is nontrivial.

Various generalization and analogues for the Ba\-er--Su\-zu\-ki theorem were investigated by many authors in
\cite{AlpLy,Mamont,Soz,OMS,Gu,FGG,GGKP,GGKP1,GGKP2, Palchik, Tyut,Tyut1,BS_odd,BS_Dpi}. For example, N. Gordeev,  F. Grunewald, B. Kunyavskii, and  E. Plotkin in \cite{GGKP2},
and independently P. Flavell, S. Guest, and R. Guralnick  in \cite{FGG} have shown that, if every four elements in given conjugacy class of a finite group generate a solvable subgroup,
then the conjugacy class is included in the solvable radical of the group.

The following proposition shows that one cannot replace in the Baer--Suzuki theorem  $p$ with a set of primes~$\pi$.

\begin{Prop}
\label{ex1}
Let  $m$ be a natural number. Choose a prime  $r$ and a set $\pi$ of primes so that  $r-1>m$ and  $\pi$ includes all primes less than~$r$ and does not include~$r$. Then,
in the symmetric group $G=S_r$, any  $m$ transpositions generate a $\pi$-sub\-gro\-up, while $\Oo_\pi(G)=1$.
\end{Prop}

The proposition not only shows that, in general, the fact that $x$ together with any of its conjugates generates a $\pi$-sub\-gro\-up does not guarantee that $x$ lies in $\Oo_\pi(G)$.
It also shows that there does not exist $m$ such that, for every set of primes $\pi$ and for every finite group $G$, the equality
$$
\Oo_\pi(G)=\{x\in G\mid \la x_1,\ldots,x_m\ra \text{ is a }\pi\text{-group for every }x_1,\ldots,x_m\in x^G\}
$$
holds.

However we show that a weaker analogue of the Baer--Suzuki theorem for the $\pi$-radical  holds. The goal of this paper is to prove Theorems~\ref{t2} and~\ref{t3} below.

\begin{theorem}\label{t2}
Let $\pi$ be a set of primes. Then there exists a natural $m$ (depending on $\pi$) such that for every finite group $G$ the equality
$$
\Oo_\pi(G)=\{x\in G\mid \la x_1,\ldots,x_m\ra \text{ is a }\pi\text{-group for every }x_1,\ldots,x_m\in x^G\}
$$
holds, i.e. $x\in\Oo_\pi(G)$ if and only if every $m$ conjugates of $x$ generate a $\pi$-sub\-gro\-up.
\end{theorem}

Similarly to~\cite[Definition~1.15]{GGKP}, given a set  of primes $\pi$,  let $\bs(\pi)$ be the minimal $m$ such that the conclusion of Theorem~\ref{t2} remains valid,
i.~e. the minimal $m$ such that, for every finite group $G$ and its element  $x$, if any  $m$ conjugates of  $x$ generate a $\pi$-sub\-gro\-up then  $x\in \Oo_\pi(G)$.
In other words  $m<\bs(\pi)$ if and only if there exists a group  $G$ and  $x\in G\setminus \Oo_\pi(G)$ such that any $m$ conjugates of  $x$ generate a $\pi$-sub\-gro\-up. Replacing
$G$ with $G/\Oo_\pi(G)$ and taking the image of $x$ in $G/\Oo_\pi(G)$, we see that  $\bs(\pi)$ can be defined by the condition that if $m<\bs(\pi)$ then there exists  $G$ with
$\Oo_\pi(G)=1$ and a nonidentity $x\in G$ such that every  $m$ its conjugates generate a  $\pi$-sub\-gro\-up.

The conclusion of Theorem~\ref{t2} is evident if $\pi$ is the set of all primes. In this case, we can take $m$ in the statement of Theorem~\ref{t2} to be equal to $1$ (and even  $0$).
The Baer--Su\-zu\-ki theorem implies that if $\pi=\{p\}$, then $\bs(\pi)=2$. V.N.\,Tyutyanov \cite{Tyut1} and, much later, the second author in~\cite{BS_odd} showed that if  $2\notin\pi$
then $\bs(\pi)=2$. In the paper, we prove the following

\begin{theorem}\label{t3}
Let $\pi$ be a proper subset of the set of all primes. Then
 $$
r-1\leq \bs(\pi)\leq \max\{11,2(r -2)\},
$$
where $r$ is the minimal prime not in~$\pi$.
\end{theorem}

The lower bound in Theorem  \ref{t3} follows by Proposition~\ref{ex1}.

In order to prove Theorem \ref{t2} and  obtain the upper bound in Theorem  \ref{t3}, we use the reduction to almost simple groups obtained in  \cite{BS_odd} (see Lemma~\ref{red}), and the results by R.Guralnick and J. Saxl~\cite{GS}. They obtained, for every finite simple group $L$ and every automorphism  $x\in \Aut(L)$ of prime order, upper bounds\footnote{For sporadic groups these bounds were substantially improved in~\cite{DiMPZ}.} on~$\alpha(x,L)$ defined as follows.


\begin{definition} \cite{GS}\label{alphaxL}   Let  $L$ be a nonabelian simple group and  $x$ its nonidentity automorphism. Let $\alpha(x)=\alpha(x,L)$ be the minimal number  of $L$-conjugates of~$x$ which generate the group  $\la x, \Inn(L)\ra$.
\end{definition}

Similarly to Definition \ref{alphaxL}, we introduce the number $\beta_r(x,L)$ which also plays an important role in this paper.

\begin{definition} Let  $r$ be a prime divisor of the order of a nonabelian simple group $L$. For a nonidentity automorphism  $x$ of $L$, denote by
$\beta_r(x,L)=\beta_r(x)$ the minimal number of
$L$-conjugates of $x$ which generate a subgroup of order divisible by~$r$.
\end{definition}

It follows immediately from the definitions, that, for every nonabelian simple group $L$ and every nonidentity $x\in\Aut(L)$, if $r$ divides $|L|$ then the inequality
$$
\beta_r(x,L)=\beta_{r}(x)\leq \alpha(x,L)
$$
holds.
The results from  \cite{BS_odd} imply that $\beta_2(x,L)\leq2$ for any nonabelian simple group  $L$ and its nonidentity automorphism~$x$.

In this paper, we always assume that   $r$ is an odd prime.  We provide a rather rough upper bound for $\beta_r(x,L)$  depending on  $r$ only, where $L$ is a simple group of order divisible by $r$, and  $x\in \Aut(L)^\sharp$. It follows by definition that if  $y\ne 1$ is a power of~$x$, then $\beta_r(x,L)\leq\beta_r(y,L)$. In particular, it is enough to find upper bounds for  $\beta_r(x,L)$ in the case when the order of  $x$ is prime. We derive Theorem~\ref{t2} and the upper bound for $\bs(\pi)$ in Theorem~\ref{t3} from the following statement which is the main result of this paper.

\begin{theorem}\label{t4}
Let $r$ be an odd prime,  $L$ a nonabelian simple group, and let  $x\in\Aut(L)$ be of prime order. Then one of the following statements is true.
\begin{itemize}
  \item[$(1)$] $\alpha(x,L)\leq 11$;
  \item[$(2)$] $L$ is isomorphic to an alternating or classical group of Lie type, the order of  $L$ is not divisible by  $r$, and
  $\alpha(x,L)\leq r-1$;
  \item[$(3)$] $L$ is isomorphic to an alternating or  classical group of Lie type, the order of  $L$ is divisible by  $r$, and
  $\beta_r(x,L)\leq 2(r-2)$.
 \end{itemize}
In particular, if  $r$ divides   $|L|$ then $\beta_r(x,L)\leq\max\{11, 2(r-2)\}$, and if $r$ does not divide~$|L|$, then $\alpha(x,L)\leq \max\{11, r-1\}\leq \max\{11, 2(r-2)\}$.
\end{theorem}

The number $11$ in Theorem~\ref{t4} (and, as a consequence, in Theorem~\ref{t3}) is a uniform bound for  $\alpha(x,L)$ in the case when  $L$ is a sporadic or exceptional Lie type group. The existence of such a bound follows by the results in~\cite{GS}. It is likely that a detailed investigation of  $\beta_r(x,L)$ in sporadic, exceptional, and classical Lie type groups of small ranks allows one to reduce or even remove the number $11$ from both Theorems~\ref{t3} and~\ref{t4}.  We think the bound  $2(r-2)$ for $\bs(\pi)$ and $\beta_r(x,L)$ is also too big. The authors do not know any counterexamples to the following statements (notice that the first is a corollary to the second).

\begin{Conj}
Let $\pi$ be a proper subset of the set of all primes containing at least two elements, and let $r$ be the minimal prime not in~$\pi$. Then $$\bs(\pi)=\left\{\begin{array}{rl}
                               r, & \text{ if } r\in\{2,3\}, \\
                               r-1, &   \text{ if } r\geq 5.
                             \end{array}\right.
$$
\end{Conj}

\begin{Conj}
For any nonabelian simple group $L$ of order divisible by~$r$ and its every automorphism $x$ of prime order, we have
$$\beta_r(x,L)\leq\left\{\begin{array}{rl}
                               r, & \text{ if } r\in\{2,3\}, \\
                               r-1, &   \text{ if } r\geq 5.
                             \end{array}\right.
$$
\end{Conj}

In the  case $r=2$, as we have noted above, both conjectures are true. There are examples  (see~\cite[Example~2]{BS_odd}, and Proposition~\ref{beta_A_n_prop} below) showing that, for $r=3$, the value $\beta_r(x,L)$ can be equal to~$3$. Thus the bound on $\bs(\pi)$ follows from the bounds for $\beta_r(x,L)$ for a nonabelian simple~$L$ and an odd prime~$r$. In the case of alternating groups, the sharp bound gives

\begin{Prop} \label{beta_A_n_prop}
Let  $L=A_n$, $n\geq 5$, let $r\leq n$ be a prime, and let $x\in \Aut (L)$ be of prime order. Then
\begin{itemize}
\item[$(1)$] $\beta_r(x,L)=r-1$ if $x$ is a transposition;
\item[$(2)$] if $r=3$, $n=6$, and $x$ is an involution not lying in $S_6$ then   $\beta_r(x,L)= 3$;
\item[$(3)$] $\beta_r(x,L)\leq r-1$ for all other $x$.
\end{itemize}
\end{Prop}

Recall that the class of finite groups~$\mathfrak{X}$ is called {\em  radical} if in every finite group $G$ there exists {\em $\mathfrak{X}$-radical} $G_{\mathfrak{X}}$, i.e. the largest normal $\mathfrak{X}$-sub\-gro\-up \footnote{As usual, a group $X\in\mathfrak{X}$ is called an {\em $\mathfrak{X}$-gro\-up}. The largest normal $\mathfrak{X}$-sub\-gro\-up is the normal  $\mathfrak{X}$-sub\-gro\-up containing any other normal  $\mathfrak{X}$-sub\-gro\-up.}. According to~\cite[Definition~1.15]{GGKP}, let the {\em  Baer-Suzuki width} $\bs(\mathfrak{X})$ for a radical class $\mathfrak{X}$ be defined as the exact lower bound for the set of all natural numbers  $b$ such that, for every finite group $G$, the  $\mathfrak{X}$-ra\-di\-cal $G_{\mathfrak{X}}$  is equal to
$$\{x\in G\mid \langle x_1,\dots,x_b\rangle\in\mathfrak{X}\text{ for every }  x_1,\dots,x_b\in x^G\}.$$

N. Gordeev, F. Grunewald, B. Kunyavskii, E. Plotkin state the problem~\cite[Problem~1.16]{GGKP}: for what radical classes~$\mathfrak{X}$ the inequality $\bs(\mathfrak{X})<\infty$ holds? The results of this paper show that, for every set of primes $\pi$, the class of all  $\pi$-groups has finite Baer-Suzuki width. Moreover, we believe that the results of the paper make a substantial progress toward the solution of~\cite[Problem~1.16]{GGKP} in general. 

\section{Preliminaries}

\subsection{Reduction to almost simple groups and general lemmas}

\begin{definition}
\cite{BS_odd}  Let $\pi$ be a set of primes and  $m$ be a positive integer. We say that a finite group $G$ {\it lies in  ${\mathcal {BS}}_{\pi}^{m}$} ($G\in{{\mathcal B}{\mathcal S}}_{\pi}^{m}$), if
$$
\Oo_\pi(G)=\{x\in G\mid \langle x_1,\ldots,x_m\rangle \text{ is a }\pi\text{-group for every }x_1,\ldots,x_m\in x^G\}.
$$
\end{definition}

\begin{lemma}\label{red} {\em  \cite[Lemma~11]{BS_odd}} Suppose that  not all finite groups are contained in $\BS_{\pi}^{m}$ for some  $m\geq 2$, and choose ${G\notin\BS_{\pi}^{m}}$ of minimal order. Then  $G$ possesses a subgroup  $L$ and an  element $x$ such that
\begin{itemize}
\item[$(1)$] $L\trianglelefteq G$;
\item[$(2)$] $L$ is nonabelian simple;
\item[$(3)$] $L$ is neither  $\pi$- nor $\pi'$-gro\-up, where $\pi'$ is the complement of $\pi$ in the set of all primes;
\item[$(4)$] $C_G(L)=1$;
\item[$(5)$] any  $m$ conjugates of  $x$ generate a $\pi$-gro\-up;
\item[$(6)$] $x$ has prime order;
 \item[$(7)$] $G=\la x, L\ra$.
\end{itemize}
\end{lemma}

\begin{lemma} \label{2notinpi} {\em \cite[Theorem 1]{BS_odd}} If $2\not\in\pi$, then $\BS_\pi^2$ includes all finite groups.
\end{lemma}

Let  $G$ be fixed. For  $x\in G$ and a group $K$ we say that  $x$ {\it is drawn into  $K$} ($x\leadsto K,$)
if there exists a subgroup  $H\leq G$ such that
  \begin{itemize}
   \item[$(1)$] $x\in H$;
   \item[$(2)$] there exists an epimorphism $H$ on $K$;
   \item[$(3)$] the image of  $x$ in  $K$ under the epimorphism is nontrivial.
  \end{itemize}
The following statement is immediate from the definition.

 \begin{lemma} \label{leadsto}
Let $G$ and $K$ be almost simple groups with socles   $S$ and  $L$ respectively, and assume a prime $r$ divides the orders of both  $S$ and $L$. Assume inequality $\beta_r(y,L)\leq m$ for all $y\in K^\sharp$. If  $x\leadsto K$ for some  $x\in G^\sharp$, then $\beta_r(x,S)\leq m$.
\end{lemma}

The next lemma allows to find a conjugate of the centralizer of an automorphism $x$ that is normalized, but is not centralized by~$x$.

\begin{lemma} \label{guest} {\em \cite[Lemma~15]{GuestLevy}}
Let  $G$ be a finite group and $x\in\Aut(G)$ be a $p$-element. Set  $M=C_G(x)$. Suppose that  $p$ divides $|G:M|$, and either  $M=N_G(M)$ or $Z(M) =1$. Then there exists a conjugate of $M$ that is normalized but not centralized by~$x$.
\end{lemma}

\subsection{Information on almost simple groups}

For finite simple groups we use the notations of ATLAS~\cite{Atlas}.

In Table~\ref{tab1} we collect the  information on the orders of finite simple classical groups.

\begin{table}
\caption
{Simple classical groups of Lie type over a field  $\mathbb{F}_q$}\label{tab1}
\begin{tabular}{|l|l|c|r|}\hline
$L$ & Restrictions & $|L|$& $d$  \\
\hline\hline
 \multirow{2}{*}{$A_{n-1}(q)\cong L_n(q)$}&$n\geq2$; $q>3$ for &  \multirow{2}{*}{$\frac{1}{d} q^{n(n-1)/2} \prod\limits_{i=1}^n(q^{i}-1)$}&  \multirow{2}{*}{$(n,q-1)$}\\
&  $n=2$&&\\\hline
\multirow{2}{*}{${^2A}_{n-1}(q)\cong U_n(q)$}& $n\geq3$; $q>2$   &\multirow{2}{*}{$\frac{1}{d} q^{n(n-1)/2} \prod\limits_{i=1}^n(q^{i}-(-1)^{i})$} &\multirow{2}{*}{$(n,q+1)$}\\
& for $n=3$ &&\\  \hline
$B_n(q)\cong O_{2n+1}(q)$& $n\geq3$ &  $\frac{1}{d} q^{n^2} \prod\limits_{i=1}^n(q^{2i}-1)$& $(2,q-1)$\\
\hline
\multirow{2}{*}{$C_n(q)\cong  S_{2n}(q)$}& $n\geq2$; $q>2$  for &\multirow{2}{*}{ $\frac{1}{d} q^{n^2} \prod\limits_{i=1}^n(q^{2i}-1)$} & \multirow{2}{*}{$(2,q-1)$}\\
&  $n=2$&&\\\hline
$D_n(q)\cong O^+_{2n}(q)$& $n\geq4$ &$\frac{1}{d} q^{n(n-1)}(q^n-1) \prod\limits_{i=1}^{n-1}(q^{2i}-1)$ & $(4,q^n-1)$ \\ \hline
${^2D}_n(q)\cong O_{2n}^-(q)$& $n\geq4$  &$\frac{1}{d} q^{n(n-1)}(q^n+1) \prod\limits_{i=1}^{n-1}(q^{2i}-1)$ & $(4,q^n+1)$\\ \hline
\end{tabular}
\end{table}

Our terminology for automorphisms of groups of Lie type agrees with that of  \cite{GLS}  and is different with that of \cite{Car}. We quote it here explicitly.

The definition of inner-diagonal automorphisms is the same in \cite{Car} and in \cite{GLS}, and we use this definition. In \cite[Definition 2.5.10]{GLS} subgroups $\Phi_K$ and $\Gamma_K$ of $\Aut(K)$ are defined for arbitrary group of Lie type $K$. For groups of Lie type we usually use the letter~$L$, so the corresponding subgroups we denote by $\Phi_L$ and $\Gamma_L$. We denote the group of inner-diagonal automorphisms of $L$ by  $\L$ or $\Inndiag(L)$.

\begin{lemma}\label{Aut} {\em \cite[Theorem 2.5.12]{GLS}} Let $L$ be a simple group of Lie type over a field  ${\mathbb F}_q$ of characteristic $p$. Then  $\Aut(L)$ is a split extension  of $\L$ by an abelian group  $\Phi_L\Gamma_L$. Moreover  $\Phi_L\Gamma_L\cong\Phi_L\times\Gamma_L$, except the following cases:
\begin{itemize}
\item[$(1)$] $L=B_2(q)$, $q$ is a power of  $2$ and $\Phi_L\Gamma_L$ is cyclic with  $|\Phi_L\Gamma_L:\Phi_L|=2$;
\item[$(2)$] $L=F_4(q)$, $q$ is a power of $2$ and $\Phi_L\Gamma_L$ is cyclic with $|\Phi_L\Gamma_L:\Phi_L|=2$;
\item[$(3)$] $L=G_2(q)$, $q$ is a power of $3$ and $\Phi_L\Gamma_L$ is cyclic with $|\Phi_L\Gamma_L:\Phi_L|=2$.
\end{itemize}
\end{lemma}

An automorphism of prime order $\alpha\in\Aut(L)\setminus\L$ of  an untwisted group of Lie type $L$ is called
\begin{itemize}
\item[]{\em field modulo $\L$}, if the image of $\alpha$ in $\Aut(L)/\L$ lies in $\Phi_L\L/\L$; elements of $\Phi_L$ we call {\em canonical field automorphisms} of $L$;
 \item[] {\em graph modulo} $\L$, if $L$ is not isomorphic to $B_2(2^n)$, $F_4(2^n)$, and $G_2(3^n)$ and the image of $\alpha$ in $\Aut(L)/\L$ lies in $\Gamma_L\L/\L$; elements of  $\Gamma_L$ we call  {\it canonical graph automorphisms} of $L$;
\item[] {\em graph-field modulo} $\L$ in the remaining cases; at that elements of $\Phi_L\Gamma_L\setminus\Phi_L$ for $B_2(2^n)$, $F_4(2^n)$, and $G_2(3^n)$, and elements of  $\Phi_L\Gamma_L\setminus(\Phi_L\cup\Gamma_L)$ for the remaining untwisted groups of Lie type  we call {\it canonical graph-field automorphisms} of $L$.
\end{itemize}

Let $L$ be a twisted group of Lie type, not isomorphic to a Suzuki or a Ree group, obtaining from its  untwisted analogue as a set of stable points of  an automorphism of order $d\in\{2,3\}$ (see \cite[Ch.~13]{Car}).  In this case  $\Gamma_L=1$ (see \cite[Theorem~2.5.12]{GLS}). Consider $\alpha\in\Aut(L)\setminus\L$ of prime order. We say that $\alpha$ is
\begin{itemize}
\item[]{\em field modulo $\L$}, if the order of $\alpha$ is coprime to $d$; elements of $\Phi_L$ we call  {\em canonical field automorphisms} of $L$;
 \item[] {\em graph modulo} $\L$, if the order of $\alpha$ equals  $d$; elements of $\Phi_L$ we call {\em canonical graph automorphisms} of $L$.
 \item[] there are no graph-field automorphisms of prime order modulo $\L$.
\end{itemize}

Finally, all noninner automorphisms of a Suzuki or a Ree group  $L$, are called {\em  field modulo~$\L$}.

Notice that the notion of a field and a graph-field  modulo $\L$  automorphism $\alpha$ of $L$ coincide with the notion of a field or a graph-field automorphism in  \cite[Definition~2.5.13]{GLS}.

\begin{lemma}\label{Field_Aut} {\em \cite[Proposition~4.9.1]{GLS}} Let $L={}^d\Sigma(q)$ be a simple group of Lie type over a field  ${\mathbb F}_q$, where  $\Sigma$ is an indecomposable root system,  $d$ is either an empty symbol, or   $2$ (i.e. ${}^d\Sigma(q)\ne{}^3D_4(q)$). Let  $x$  and $y$ be automorphisms of  $L$, having the same prime order. Assume also that both  $x$ and $y$ are either field or graph-field automorphisms modulo~$\L$. Then subgroups  $\langle x\rangle$ and
$\langle y\rangle$ are conjugate under~$\L$. Moreover if $x$ is a graph-field automorphism, and ${}^d\Sigma\in\{A_{n-1},D_{n}\}$, then $|x|=2$ and ${}^2\Sigma(q^{1/2})\leq C_{L}(x)\leq \widehat{{}^2\Sigma(q^{1/2})}$.
\end{lemma}

By  $\tau$ we denote the automorphism of $\GL_n(q)$, acting by
$$\tau:A\mapsto (A^{-1})^\top,$$
where $A^\top$  is the transposed of  $A$. If $q$ is a power of a prime $p$, by $\varphi_{p^k}$ we denote an automorphism of  $\GL_n(q)$ acting by
$$\varphi_{p^k}:(a_{i,j})\mapsto (a_{i,j}^{p^k}).$$
We use the same symbols $\tau$ and $\varphi_{p^k}$ for the induced automorphisms of $\mathrm{PGL}_n(q)$, $\mathrm{SL}_n(q)$, and  $L_n(q)=\mathrm{PSL}_n(q)$. In particular,  if $q=p^k$ and $r$ divides $k$, then $\varphi_{q^{1/r}}$ is a field automorphism of order~$r$. For definiteness, we always assume that  $\mathrm{PSU}_n(q)=O^{p'}\left(C_{\mathrm{PGL}_n(q^2)}\left(\tau\varphi_{q}\right)\right)$. As usual, we use the notations $L_n^\varepsilon(q)=\mathrm{PSL}_n^\varepsilon(q)$, $\varepsilon=\pm$ for linear and unitary groups, assuming   $L_n^+(q)=\mathrm{PSL}_n^+(q)=\mathrm{PSL}_n(q)$ and $L_n^-(q)=\mathrm{PSL}_n^{-}(q)=\mathrm{PSU}_n(q)$. By  $E_k$ we denote the identity  $(k\times k)$-matrix and by $A\otimes B$ the Kronecker product of matrices $A$ and~$B$.

\begin{lemma}\label{GraphAutGLU}
Let  $L=L_n^\varepsilon(q)$ be a simple projective special linear or unitary group and  $n\geq 5$. Then the following statements hold:
\begin{itemize}
\item[{\em (1)}] If $n$ is odd, then the coset  $\L\tau$ of $\langle\L, \tau\rangle$ contains exactly one class of conjugate involutions, and every such involution normalizes, but not centralizes a subgroup  $H$ of $S$  such that  $H\cong O_n(q)$.
\item[{\em (2)}] If $n$ is even and  $q$ is odd, then the coset  $\L\tau$ of $\langle\L, \tau\rangle$ contains exactly three classes of conjugate involutions with representatives  $x_0,x_+,x_-$ such that $x_\delta$ normalizes but does not centralize  $H_\delta$,  where  $\delta\in \{0,+,-\}$
    $$
    H_\delta\cong\left\{
    \begin{array}{cc}
      S_n(q), &\text{if } \delta=0, \\
      O^+_n(q), &\text{if } \delta=+, \\
      O^-_n(q), &\text{if } \delta=-.
    \end{array}
    \right.
    $$
\item[{\em (3)}] If both  $n$ and $q$ are even, then the coset $\L\tau$ of $\langle\L, \tau\rangle$ contains exactly two classes of conjugate involutions, and every such involution normalizes, but not centralizes a subgroup isomorphic to~$S_n(q)$.
\end{itemize}
\end{lemma}

\begin{proof}[Proof]
The statement on the centralizers and the number of $\L$-conjugacy classes of involutions for  $q$ odd follows by \cite[Theorem~4.5.1 with Table~4.5.1]{GLS}, while for $q$ even by \cite[(19.8)]{AsSeitz}. If $n$ is odd, then the socle of the centralizer in $\L$ of a graph involution is isomorphic to $H$, while for $n$ even and $q$ odd the socle of the centralizer in $\L$ is isomorphic to $H_\delta$, in particular, the centralizers of  $x_0$, $x_+$, and~$x_-$ are pairwise nonisomorphic. Also the centers of centralizers are trivial. Since for $n\geq 5$ these centralizers have even indices, by Lemma~\ref{guest} every such involution normalizes, but not centralizes a subgroup conjugate to its centralizer. Thus we obtain statements (1) and~(2) of the lemma.

Now assume that  both $n$ and $q$ are even.
In this case we can take  $\tau I$ as a representative of one of the conjugacy classes of involutions, where $I$ is the projective image of a block-diagonal matrix with blocks $\left(\begin{array}{rr}0&1\\1&0 \end{array}\right)$ on the diagonal, i.~e.
$$
I=\left(
\begin{array}{rr}
  0 & 1 \\
  1 & 0
\end{array}
\right)\otimes E_{n/2}.
$$
It is easy to see that  $C_L(x)\cong S_n(q)$ (if  $L$ is linear, direct computations show that  $C_L(x)$ consists exactly of projective matrices  $A$ satisfying $A^\top IA=I$, and these matrices form  $S_n(q)$; if  $L$ is unitary, see   \cite[(19.8)]{AsSeitz}). Since the index  $\vert L:C_L(x)\vert$ is even and $Z(C_L(x))=1$, by Lemma~\ref{guest} there exists a subgroup  $M$ of  $S$, isomoprphic to  $C_L(x)$, such that  $x$ centralizes, but not normlizes $M$. In order to obtain the representative  $y$ of the second conjugacy class, it is enough to take any transvection  $t\in C_L(x)$ and set $y=\tau I t$. By construction,  $y$ normalizes but not centalizess $C_L(x)\cong S_n(q)$.
\end{proof}

\begin{lemma}\label{Graph_Inv_D_n} Let $V$ be a vector space over a field $\mathbb{F}_q$ of odd order $q$,  $\dim V=2n$  for $n\geq 5$, and   $V$ is equipped with a nondegenerate symmetric bilinear form of the sign $\varepsilon\in\{+,-\}$. Let  ${\rm O}$ be the full group of isometries and   $\Delta$ be the full group of similarities of $V$, ${\rm SO}$ be the subgroup of  ${\rm O}$ consisting of elements with determinant  $1$,  $\Omega={\rm O}'$. Denote by $$\overline{\phantom{x}}:\Delta\rightarrow\Delta/Z(\Delta)$$
the canonical epimorphism. Let $L=\overline{\Omega}=O_{2n}^\varepsilon(q)$. Then the following statements hold.

\begin{itemize}
  \item[$(1)$] A canonical graph automorphism  $\overline{\gamma}$ of $L$ is contained in  $\overline{{\rm O}}\setminus \overline{{\rm SO}}$, while  $\overline{\Delta}$ coincides with $\langle\L,\overline{\gamma}\rangle$.
  \item[$(2)$] All graph modulo $\L$ involutions are images of involutions from  $\Delta$.
  \item[$(3)$] If  $n$ is even, then there are  $n/2$ classes of $\L$-conjugate graph modulo $\L$ involutions with representatives  $\overline{\gamma}_1=\overline{\gamma}, \overline{\gamma}_2,\dots,\overline{\gamma}_{n/2}$, where $\gamma_i$ for  $i=1,\dots, n/2$ is an involution in ${\rm O}$ such that the eigenvalue $-1$ of $\gamma_i$ has multiplicity ${2i-1}$. Every $\overline{\gamma}_i$ normalizes but does not centralize a subgroup of  $L$ isomorphic to $O_{2n-1}(q)$.

  \item[$(4)$] If $n$ is odd, then there exist  $(n+3)/2$  classes of  $\L$-conjugate graph modulo $\L$ involutions with representatives $\overline{\gamma}_1=\overline{\gamma}, \overline{\gamma}_2,\dots,\overline{\gamma}_{(n+1)/2}$ and $\overline{\gamma}_{(n+1)/2}'$, where for $i=1,\dots, (n+1)/2$ $\overline{\gamma}_i$ is an involution from  ${\rm O}$ such that the eigenvalue  $-1$ of $\overline{\gamma}_i$ has multiplicity  ${2i-1}$, while  $\overline{\gamma}'_{(n+1)/2}$ is an involution from  $\Delta\setminus {\rm O}$ such that its eigenvalue  $-1$ has multiplicity $n$. Every  $\overline{\gamma}_i$ normalizes but does  not centralize a subgroup of $L$ isomorphic to $O_{2n-1}(q)$, while $\overline{\gamma}_{(n+1)/2}'$ normalizes but not centralizes a subgroup of  $L$ isomorphic to~$O_{n}(q^2)$.
\end{itemize}
\end{lemma}

\begin{proof} All statements of the lemma, except the existence of subgroups normalized but not centralized by corresponding involutions follow by \cite[Theorems~4.3.1 and 4.3.3, Tables~ 4.3.1 and 4.3.3, and Remark~4.5.4]{GLS}.

We show that for  the involutions $\gamma_i$ there exists a nondegenerate invariant subspace $U$ of $V$ of $\dim U=2n-1$ with nonscalar action. Since  $\gamma_i$ is an isometry, the subspaces  $V_+$ and $V_-$, consisting of eigenvectors corresponding to eigenvalues  $1$ and  $-1$ respectively, are orthogonal and  $V= V_+\oplus V_-$. So  $V_+$ and $V_-$ are nondegenerate  $\gamma_i$-stable subspaces, and  $\gamma_i$ acts on both of them as a scalar multiplication. Let  $u\in V_+$ be a nonsingular vector. Then the subspace  $W=u^\perp$ is  $\gamma_i$-stable. Since  $\dim V_+=2n-2i+1>1$, the restriction of $\gamma_i$ on $W$ has two eigenvalues $1$ and~$-1$, and so its action on $W$ is nonscalar. Hence  $\overline{\gamma}_i$ normalizes but not centralizes the derived subgroup $I(W)'\cong O_{2n-1}(q)$ of the group of all isometries of~$W$.

By~\cite[Table~4.5.1]{GLS}  the centralizer of  $\overline{\gamma}_{(n+1)/2}'$ in $\L$ is isomorphic to the group of inner-diagonal automorphisms of $O_{n}(q^2)$. The centralizer has a trivial center and even index in~$\L$. By Lemma~\ref{guest} we conclude that $\overline{\gamma}_{(n+1)/2}'$ normalizes but not centralizes a subgroup of $L$ isomorphic to~$O_{n}(q^2)$.
\end{proof}

\begin{lemma}\label{Graph_Inv_D_n_even_q}  Assume  $L=D_n(q)=O_{2n}^\varepsilon (q)$ for $\varepsilon\in\{+,-\}$, $n\geq 4$, and even~$q$. Then every graph modulo  $\L$ involution of $\Aut(L)$  normalizes a subgroup of $L$ isomorphic to $O_{2n-2}^\eta(q)\times O_2^{\varepsilon\eta}(q)$ for $\eta\in\{+,-\}$ and does not centralizes the component $O_{2n-2}^\eta(q)$ in this product.
\end{lemma}

\begin{proof} Denote by $V$ a vector space of dimension $2n$ over $\mathbb{F}_q$ equipped with a nondegenerate quadratic form of sign~$\varepsilon$. We identify $L$ and  $\Omega(V)={\rm O}(V)'$, where ${\rm O}(V)$ is the isometry group of~$V$.
 Let $t$ be a graph modulo $\L$ involution of $\Aut(L)$. It is a well-known fact that $t$ belongs to ${\rm O}(V)$ (see, for example, \cite[p.~34--35]{Bray} and \cite[\S~2.7--2.8]{KL}). 
 In view of \cite[(7.6) and (8.10)]{AsSeitz}, since $t\not \in \Omega(V)$, the rank of $t-1$ is odd. Now by \cite[(7.5)(1)]{AsSeitz} we obtain that $t$ stabilizes a decomposition $V=Y\oplus Y^\perp$, where $Y$ is nondegenerate of dimension~$2$. Therefore $t$ normalizes $\Omega(Y)\times \Omega(Y^\perp)$. If $t$ does not centralize $\Omega(Y^\perp)$, we obtain the lemma. If $t$ centralizes $\Omega(Y^\perp)$, then $t$ acts identically on $Y^\perp$, i.~e. $Y^\perp$ consists of $t$-stable vectors. In this case we can take any nondegenerate $2$-dimensional subspace $U$ of $Y^\perp$. By construction, $t$ stabilizes $U\oplus U^\perp$ and acts nontrivially on $U^\perp$. Hence $t$ normalizes $\Omega(U)\times\Omega(U^\perp)$ and does not centralizes~$\Omega(U^\perp)$.
\end{proof}

\begin{lemma}\label{alpha_A_n} {\em \cite[Lemma~6.1]{GS}} Assume $L=A_n$, where $n\geq 5$ is a simple alternating group. Then  $\alpha(x,L)\leq n-1$ for any nonidentity element  $x\in\Aut(L)$. Moreover if  $n\ne 6$ and $x$ is not a transposition, then~$\alpha(x,L)\leq n/2$.
\end{lemma}

Notice that  $A_6\cong L_2(9)$ so for $L=A_6$ the information on  $\alpha(x,L)$ is provided in Lemma~\ref{alpha_classic} below.

\begin{lemma}\label{Sporadic} {\em \cite[Table~1]{GS}} Let  $L$ be a simple sporadic group. Then  $\alpha(x,L)\leq 8$ for every nonidentity element  $x\in\Aut(L)$.
\end{lemma}

\begin{lemma}\label{Exceptional} {\em \cite[Theorem~5.1]{GS}} Let $L$ be a simple exceptional group of Lie type. Then $\alpha(x,L)\leq 11$ for every nonidentity $x\in\Aut(L)$.
\end{lemma}

\begin{lemma} \label{SemisimpleInvariant} Let $V$ be a vector space of dimension  $n> 2$ over a finite field  $F$, equipped with either trivial, or nondegenerate bilinear or unitary form. Let $x$ be a nonscalar similarity of $V$, and assume  that the characteristic of $F$ does not divide~$|x|$. Assume also that  $x$ possesses a $1$-dimensional invariant subspace~$U$. Then there exists an $x$-invariant subspace $W$ of codimension  $1$ such that $x$ is nonscalar on it. Moreover, if $U$ is nondegenerate then $W$ can be chosen nondegenerate.
\end{lemma}

\begin{proof} Let   $W$ be either an  $x$-invariant complement of $U$, if the form is trivial (the existence of such a complement follows by the Maschke theorem), or  $W=U^\perp$, if the form is nondegenerate. In any case,  $W$ is  $x$-invariant of codimension~1 and $W$ is nondegenerate if $U$ is nondegenerate.

If $x$ is nonscalar on $W$, we obtain the lemma. Otherwise, we show that there is an $x$-invariant subspace $W_0$ of codimension 1 such that $W_0\ne W$. This implies that $V=W+W_0$ and $W\cap W_0\ne 0$, because $n>2$. Since $x$ is nonscalar on $V$, we conclude that $x$ is nonscalar on~$W_0$. Furthermore, we show that $W_0$ can be chosen nondegenerate if $U$ is nondegenerate.

  If $x$ is scalar  on $W$ then every $1$-dimension subspace of $W$ is $x$-invariant. Moreover, if the form is nondegenerate, one can choose an $1$-dimension subspace $U_0$ of $W$ such that $U_0^\perp\ne W$ (otherwise $W$ would be totally isotropic which contradicts the Witt lemma).   Moreover, if $U$ is nondegenerate, then the form is not symplectic. Therefore, nondegenerate subspace $W$ possesses a nonisotropic vector, and $U_0$ can be chosen nondegenerate.
Choose $W_0$ to be an $x$-invariant complement of $U_0$, if the form is trivial, or $W_0=U_0^\perp$, if the form is nondegenerate. It is easy to see that $W_0$ satisfies the conclusion of Lemma~\ref{SemisimpleInvariant}. 
\end{proof}

 \begin{center}\begin{table}
\caption
{Bounds on $\alpha(x,L)$ for classical groups $L$}\label{tab2}
\begin{center}\begin{tabular}{|c|c|c|c|r|}\hline
    \multirow{2}{*}{ $L$ }   & conditions & conditions & conditions  & \multirow{2}{*}{$\alpha(x,L)$} \\
  & on $n$   & on $q$  &  on $x$     &   \\
\hline\hline
\multirow{18}{*}{$A_{n-1}(q)\cong L_n(q)$} & \multirow{10}{*}{$n=2$}  & \multirow{3}{*}{$q\ne 5,9$} & $|x|> 2$      & $2$   \\
                                               \cline{4-5}
                         &        &         & field,  $|x|=2$       & $\leq 4$  \\
                                              \cline{4-5}
                         &        &             &\multirow{2}{*}{ not field, $|x|=2$}     & \multirow{2}{*}{$3$}\\
                                  \cline{3-3}
                         &        & \multirow{4}{*}{$q= 9$}  &     & \\
                                             \cline{4-5}
                         &        &        & field,   $|x|=2$    & $5$     \\
                                             \cline{4-5}
                         &        &        & $|x|=3$       & $3$     \\
                                             \cline{4-5}
                         &        &       & $|x|>3$       & $2$    \\
                                  \cline{3-5}
                         &        & \multirow{3}{*}{$q=5$}    &  $|x|>2$      & $2$  \\
                                             \cline{4-5}
                         &        &          & not diagonal, $|x|=2$   & $3$  \\
                                             \cline{4-5}
                         &        &          & diagonal,  $|x|=2$      & $4$  \\
                                  \cline{3-5}

                         \cline{2-5}
                         & \multirow{3}{*}{ $n=3$} &          & not graph-field      & \multirow{2}{*}{$\leq 3$} \\
                         &        &          & or $|x|\ne 2$       &   \\
                                             \cline{4-5}
                         &        &          &    graph-field,  $|x|=2$     & $\leq 4$  \\
                         \cline{2-5}
                         &  \multirow{4}{*}{$n=4$} &  \multirow{2}{*}{$q>2$}  &graph     & $\leq 6$  \\
                                            \cline{4-5}
                         &        &        &       \multirow{2}{*}{not graph}      &  \multirow{2}{*}{$\leq 4$}  \\
                                  \cline{3-3}
                         &        & \multirow{2}{*}{$q=2$}  &      & \\
                                           \cline{4-5}
                         &        &        &    graph     & $7$  \\
                         \cline{2-5}
                         &  \multirow{1}{*}{$n>4$} &          &               & $\leq n$  \\
                                  \hline   \hline
\multirow{11}{*}{${^2A}_{n-1}(q)\cong U_n(q)$} & \multirow{4}{*}{$n=3$}      &$q>3$    &                      &    $\leq 3$  \\
                                          \cline{3-5}
                             &            &\multirow{3}{*}{$q=3$}    & not inner        &    \multirow{2}{*}{$\leq 3$}  \\
                             &            &         & or $|x|\ne 2$       &              \\
                                                    \cline{4-5}
                             &            &         &inner, $|x|=2$   & $4$           \\
                             \cline{2-5}
                             & \multirow{6}{*}{$n=4$}      &  \multirow{2}{*}{$q>2$}   & not graph         & $\leq 4$  \\
                                                     \cline{4-5}
                             &           &           &     \multirow{2}{*}{graph }        &  \multirow{2}{*}{$\leq 6$ } \\
                                         \cline{3-3}
                             &           & \multirow{4}{*}{$q=2$}      &           &   \\  \cline{4-5}
                             &           &        & transvection         & $\leq 5$       \\
                                                    \cline{4-5}
                             &           &           & not transvection      & \multirow{2}{*}{$\leq 4$}  \\

                             &           &       & or not graph     &      \\
                             \cline{2-5}
                             &  $n>4$    &           &                    & $\leq n$  \\
                                  \hline   \hline
\multirow{7}{*}{$C_{n}(q)\cong S_{2n}(q)$} &  \multirow{4}{*}{ $n=2$} &    \multirow{2}{*}{ $q>3$}   & $|x|= 2$           & $\leq 5$  \\
                                                     \cline{4-5}
                             &           &             &  \multirow{2}{*}{$|x|> 2$ }          &  \multirow{2}{*}{$\leq 4$ }   \\
                                         \cline{3-3}
                             &           & \multirow{2}{*}{$q=3$}     &           &  \\
                                                     \cline{4-5}
                             &           &           & $|x|= 2$           & $\leq 6$   \\
                             \cline{2-5}
                             &  \multirow{3}{*}{$n>2$}    &           &  not transvection    & $\leq n+3$  \\
                                         \cline{3-5}
                             &           & odd   &    \multirow{2}{*}{  transvection  }  & $\leq 2n$  \\
                                         \cline{3-3} \cline{5-5}
                             &           & even     &       & $\leq 2n+1$  \\
                           \hline   \hline
 \multirow{4}{*}{$B_{n}(q)\cong O_{2n+1}(q)$} & \multirow{4}{*}{$n\geq 3$}  & \multirow{2}{*}{odd} &   reflection      &  $ 2n+1$    \\
                                                    \cline{4-5}
                             &            &         &  not reflection        & \multirow{2}{*}{$\leq n+3$} \\
                                          \cline{3-4}
                             &            & \multirow{2}{*}{even}   &    not transvection &  \\
                                                     \cline{4-5}
                             &            &         &     transvection    & $2n+1$  \\
                             \hline   \hline

  \multirow{2}{*}{$D_{n}(q)\cong O_{2n}^+(q)$,} & \multirow{4}{*}{$n\geq 4$}  & \multirow{2}{*}{odd} &   reflection &  $ 2n$    \\
                                                    \cline{4-5}
                             &            &         &  not reflection        & \multirow{2}{*}{$\leq n+3$} \\
                                          \cline{3-4}
  \multirow{2}{*}{${}^2D_{n}(q)\cong O_{2n}^-(q)$}  &            & \multirow{2}{*}{even}   &    not transvection    &  \\
                                                     \cline{4-5}
                             &            &         &     transvection    & $2n$  \\
                             \hline
\end{tabular}\end{center}
\end{table}\end{center}

\begin{lemma} \label{alpha_classic}
Let  $L$ be a simple classical group and   $x$ its automorphism of prime order. Then a bound on  $\alpha(x,L)$ is given in the last column of Table~{\em\ref{tab2}}.
\end{lemma}

As a corollary to Lemma \ref{alpha_classic}, we immediately obtain

\begin{lemma}\label{Small_rank} Let  $L$ be a simple classical group of Lie type, possessing an automorphism $x$ of prime order such that $\alpha(x,L)> 11$. Then
\begin{itemize}
  \item[$(1)$] if $L=A_{n-1}(q)=L_n(q)$ or $L={}^2A_{n-1}(q)=U_n(q)$, then $n\geq12$;
  \item[$(2)$] if $L=B_{n}(q)=O_{2n+1}(q)$ or $L=C_{n}(q)=S_{2n}(q)$, then $n\geq6$;
  \item[$(3)$] if $L=D_{n}(q)=O^+_{2n}(q)$ or $L={}^2D_{n}(q)=O^-_{2n}(q)$, then $n\geq6$.
\end{itemize}
\end{lemma}

\begin{lemma} \label{alpha_irreducible}
Let  $L$ be a simple classical group and $x$ be its automorphism of prime order, induced by an irreducible similarity. Then $\alpha(x,L)\leq 3$.
\end{lemma}
\begin{proof} See \cite[proof of Theorems 4.1--4.4]{GS}. More precisely, for linear groups \cite[page~534]{GS}, for unitary groups \cite[page~536]{GS}, for symplectic groups \cite[page~538]{GS}, and for orthogonal groups \cite[page~539]{GS}.
\end{proof}

\begin{lemma} \label{Parabolic} {\em \cite[Lemma 2.2]{GS}}
Let $L$ be a simple group of Lie type, let $G = \L$ and let $x\in G$.
\begin{itemize}
\item[$(1)$] If  $x$ is unipotent, let  $P_1$ and $P_2$ be distinct maximal parabolic subgroups containing a common Borel subgroup of $G$ with unipotent radicals  $U_1$ and  $U_2$. Then $x$ is conjugate to an element of  $P_i\setminus U_i$ for  $i = 1$ or $i = 2$.
\item[$(2)$]  If  $x$ is semisimple, assume that $x$ lies  in a parabolic subgroup of~$G$. If the rank of  $L$  is at least two, then there exists a maximal parabolic subgroup~$P$ with Levi complement  $J$ such that $x$ is conjugate to and element from  $J$, not centralized by any Levi component (possibly solvable) of~$J$.
\end{itemize}
\end{lemma}

\subsection{Large  $r'$-subgroups in classical groups}

\begin{lemma} \label{r_not divi}
Let  $L$ be a simple classical group or an alternating group. Assume that an odd prime  $r$ does not divides the order of $L$.
Then the following statements hold.
\begin{itemize}
  \item[$(1)$] If $L=A_n$, then $n\leq r-1$.
  \item[$(2)$] If $L=L_n(q)$, then  $n\leq r-2$ and  $ r\geq n+2.$
  \item[$(3)$] If $L=U_n(q)$, then  $n\leq r-2$ and $ r\geq n+2.$
  \item[$(4)$] If $L=S_{2n}(q)$, then  $2n\leq r-3$ and   $r\geq 2n+3.$
  \item[$(5)$] If $L=O_{2n+1}(q)$, then  $2n\leq r-3$ and   $r\geq 2n+3.$
  \item[$(6)$] If $L=O^+_{2n}(q)$ or $L=O^-_{2n}(q)$, then  $2n\leq r-1$ and    $r\geq 2n+1.$
  \item[$(7)$] If $L=L_n(q^2)$, then  $2n\leq r-3$ and   $r\geq 2n+3.$
\end{itemize}
\end{lemma}

\begin{proof} Since $r$ does not divides $|L|$, numbers $r$ and $q$ are coprime. Moreover,  by the Little Fermat theorem,
the order of $L$ (see Table~\ref{tab1}) is not divisible by~$q^{r-1}-1$. Now, using the parity of  $r-1$, we see that $q^{r-1}-1=q^{r-1}-(-1)^{r-1}$, and we obtain the statement of the lemma by using Table~\ref{tab1}.
\end{proof}

\begin{lemma} \label{non_parabolic_divisor}
  Assume,  $L$ is a simple classical group of Lie type satisfying one of the following conditions:
      \begin{itemize}
        \item $L=A_{n-1}(q)=L_n(q)=L_n^+(q)$, $n\geq 2$;
        \item $L={}^2A_{n-1}(q)=U_n(q)=L_n^-(q)$, $n\geq 3$;
        \item $L=B_{n}(q)=O_{2n+1}(q)$, $n\geq 3$;
        \item $L=C_{n}(q)=S_{2n}(q)$, $n\geq 2$;
        \item $L=D_{n}(q)=O^+_{2n}(q)$, $n\geq 4$;
        \item $L={}^2D_{n}(q)=O^-_{2n}(q)$, $n\geq 4$.
      \end{itemize}
       Suppose that an odd prime  $r$ does not divides the order of a maximal parabolic subgroup of~$L$.
Then one of the following statements hold:
\begin{itemize}
 \item[$(1)$]  $L=L_n(q)=L_n^+(q)$ and $r\geq \left[\displaystyle\frac{n+1}{2}\right]+2\geq\displaystyle\frac{n+4}{2}$;
 \item[$(2)$]  $L=U_n(q)=L_n^-(q)$ and $r\geq \left[\displaystyle\frac{n+1}{2}\right]+2\geq\displaystyle\frac{n+4}{2}$;
 \item[$(3)$]  $L=S_{2n}(q)$ and $r\geq  \displaystyle\frac{2n+7}{3}$;
 \item[$(4)$]  $L=O_{2n+1}(q)$ and $r\geq  \displaystyle\frac{2n+7}{3}$;
 \item[$(5)$]  $L=O^+_{2n}(q)$  and $r\geq  \displaystyle\frac{2n+5}{3}$;
 \item[$(6)$]  $L=O^-_{2n}(q)$ and $r\geq  \displaystyle\frac{2n+5}{3}$.
 \end{itemize}
\end{lemma}

\begin{proof} Denote by  $P$ a maximal parabolic subgroup of~$L$ of order not divisible by~$r$. The Levi factor of $P$ has at most two components, corresponding to the connected components of the Dynkin diagram with one node removed. The number $r$ does not divide the orders of the components. Assume that we remove a root $r_m$ for  $A_{n-1}(q)$, $B_n(q)$, $C_n(q)$, and $D_n(q)$, or the root  $r_m^1$ for ${}^2A_{n-1}(q)$, and ${}^2D_{n-1}(q)$ on pic.~\ref{p90}--\ref{p92}\footnote{In view of symmetry reasons we may assume that  $m\ne n$ in case $D_n(q)$.}. Then we obtain one of the following cases\footnote{We agree that $O_6^+(q)=D_3(q)=A_3(q)=L_4(q)$, $O_4^+(q)=D_2(q)=A_1(q)\times A_1(q)=L_2(q)\times L_2(q)$, $O_6^-(q)={}^2D_3(q)={}^2A_3(q)=U_4(q)$, $O_4^-(q)={}^2D_2(q)=A_1(q^2)=L_2(q^2)$. }:
 \begin{itemize}
   \item[(a)] $L=A_{n-1}(q)=L_n(q)$ and $r$ does not divide the orders of $A_{m-1}(q)=L_m(q)$ and $A_{n-1-m}(q)=L_{n-m}(q)$, where $1\leq m \leq n-1$;
   \item[(b)] $L={}^2A_{n-1}(q)=U_n(q)$ and $r$ does not divide the orders of $A_{m-1}(q^2)=L_m(q^2)$ and ${}^2A_{n-1-2m}(q)=U_{n-2m}(q)$, where $1\leq m \leq [n/2]$;
   \item[(c)] $L=B_n(q)=O_{2n+1}(q)$ and $r$ does not divide the orders of  $A_{m-1}(q)=L_m(q)$ and $B_{n-m}(q)=O_{2(n-m)+1}(q)$, where $1\leq m \leq n-1$;
   \item[(d)] $L=C_n(q)=S_{2n}(q)$ and $r$ does not divide the orders of $A_{m-1}(q)=L_m(q)$ and $C_{n-m}(q)=S_{2(n-m)}(q)$, where $1\leq m \leq n-1$;
   \item[(e)] $L=D_n(q)=O^+_{2n}(q)$ and $r$ does not divides the orders of $A_{m-1}(q)=L_m(q)$ and $D_{n-m}(q)=O^+_{2(n-m)}(q)$, where $1\leq m \leq n-1$;
   \item[(f)] $L={}^2D_n(q)=O^-_{2n}(q)$ and $r$ does not divide the orders of $A_{m-1}(q)=L_m(q)$ and ${}^2D_{n-m}(q)=O^-_{2(n-m)}(q)$, where $1\leq m \leq n-1$.
 \end{itemize}

Consider all these cases separately.

\begin{figure}
\begin{center}
\unitlength=8mm \begin{picture}(4,1)(3,0)
   \put(0,0){\circle {0.2}}   \put(0,0.5){\makebox(0,0){$r_1$}}
   \put(2,0){\circle {0.2}}  \put(2,0.5){\makebox(0,0){$r_2$}}
   \put(4,0){\circle {0.2}}   \put(4,0.5){\makebox(0,0){$r_3$}}
   \put(6,0){\circle {0.2}}   \put(6,0.5){\makebox(0,0){$r_{n-2}$}}
   \put(8,0){\circle {0.2}}   \put(8,0.5){\makebox(0,0){$r_{n-1}$}}
 \put(0.1,0){\line(1,0){1.8}}
 \put(2.1,0){\line(1,0){1.8}}
\put(4.1,0){\line(1,0){0.5}}
\put(5.4,0){\line(1,0){0.5}}
\put(5,0){\makebox(0,0){$\cdots$}}
\put(6.1,0){\line(1,0){1.8}}
\end{picture}
\caption{\label {p90}   
$A_n(q)$. }
\end{center}
\end{figure}

\vspace{8cm}

\begin{figure}
\begin{center}
\unitlength=8mm \begin{picture}(4,5)(3,0)
   \put(0,5){\circle {0.2}}   \put(0,5.5){\makebox(0,0){$r_1$}}
   \put(2,5){\circle {0.2}}  \put(2,5.5){\makebox(0,0){$r_2$}}
   \put(4,5){\circle {0.2}}   \put(4,5.5){\makebox(0,0){$r_3$}}
   \put(6,5){\circle {0.2}}   \put(6.5,5.5){\makebox(0,0){$r_{k-1}$}}
   \put(8,5){\circle {0.2}}   \put(8.5,5.5){\makebox(0,0){$r_{k}$}}
 \put(0.1,5){\line(1,0){1.8}}
 \put(2.1,5){\line(1,0){1.8}}
\put(4.1,5){\line(1,0){0.5}}
\put(5.4,5){\line(1,0){0.5}}
\put(5,5){\makebox(0,0){$\cdots$}}
\put(6.1,5){\line(1,0){1.8}}
   \put(0,3){\circle {0.2}}   \put(0,3.5){\makebox(0,0){$r_{n-1}$}}
   \put(2,3){\circle {0.2}}  \put(2,3.5){\makebox(0,0){$r_{n-2}$}}
   \put(4,3){\circle {0.2}}   \put(4,3.5){\makebox(0,0){$r_{n-3}$}}
   \put(6,3){\circle {0.2}}   \put(6.5,3.5){\makebox(0,0){$r_{k+2}$}}
   \put(8,3){\circle {0.2}}   \put(8.5,3.5){\makebox(0,0){$r_{k+1}$}}
 \put(0.1,3){\line(1,0){1.8}}
 \put(2.1,3){\line(1,0){1.8}}
\put(4.1,3){\line(1,0){0.5}}
\put(5.4,3){\line(1,0){0.5}}
\put(5,3){\makebox(0,0){$\cdots$}}
\put(6.1,3){\line(1,0){1.8}}
\put(8,4.9){\line(0,-1){1.8}}

\put(4,2.7){\vector(0,-1){1.7}}
   \put(0,0){\circle {0.2}}   \put(0,0.5){\makebox(0,0){$r^1_1$}}
   \put(2,0){\circle {0.2}}  \put(2,0.5){\makebox(0,0){$r^1_2$}}
   \put(4,0){\circle {0.2}}   \put(4,0.5){\makebox(0,0){$r^1_3$}}
   \put(6,0){\circle {0.2}}   \put(6,0.5){\makebox(0,0){$r^1_{k-1}$}}
   \put(8,0){\circle {0.2}}   \put(8,0.5){\makebox(0,0){$r^1_{k}$}}
 \put(0.1,0){\line(1,0){1.8}}
 \put(2.1,0){\line(1,0){1.8}}
\put(4.1,0){\line(1,0){0.5}}
\put(5.4,0){\line(1,0){0.5}}
\put(5,0){\makebox(0,0){$\cdots$}}
\put(6,0.1){\line(1,0){2}}
\put(6,-0.1){\line(1,0){2}}

\end{picture}
\caption{\label {p901}   
${}^2A_{n-1}(q)$, where $n-1=2k$, $k=\left[\displaystyle\frac{n}{2}\right]$. }
\end{center}
\end{figure}

\begin{figure}
\begin{center}
\unitlength=8mm \begin{picture}(4,7)(3,0)
   \put(0,5){\circle {0.2}}   \put(0,5.5){\makebox(0,0){$r_1$}}
   \put(2,5){\circle {0.2}}  \put(2,5.5){\makebox(0,0){$r_2$}}
   \put(4,5){\circle {0.2}}   \put(4,5.5){\makebox(0,0){$r_3$}}
   \put(6,5){\circle {0.2}}   \put(6.2,5.5){\makebox(0,0){$r_{k-1}$}}
 \put(0.1,5){\line(1,0){1.8}}
 \put(2.1,5){\line(1,0){1.8}}
\put(4.1,5){\line(1,0){0.5}}
\put(5.4,5){\line(1,0){0.5}}
\put(5,5){\makebox(0,0){$\cdots$}}
\put(6.1,4.95){\line(2,-1){1.8}}
   \put(0,3){\circle {0.2}}   \put(0,3.5){\makebox(0,0){$r_{n-1}$}}
   \put(2,3){\circle {0.2}}  \put(2,3.5){\makebox(0,0){$r_{n-2}$}}
   \put(4,3){\circle {0.2}}   \put(4,3.5){\makebox(0,0){$r_{n-3}$}}
   \put(6,3){\circle {0.2}}   \put(6.2,3.5){\makebox(0,0){$r_{k+1}$}}
   \put(8,4){\circle {0.2}}   \put(8,4.5){\makebox(0,0){$r_{k}$}}
 \put(0.1,3){\line(1,0){1.8}}
 \put(2.1,3){\line(1,0){1.8}}
\put(4.1,3){\line(1,0){0.5}}
\put(5.4,3){\line(1,0){0.5}}
\put(5,3){\makebox(0,0){$\cdots$}}
\put(6.1,3.05){\line(2,1){1.8}}
\put(4,2.7){\vector(0,-1){1.7}}
   \put(0,0){\circle {0.2}}   \put(0,0.5){\makebox(0,0){$r^1_1$}}
   \put(2,0){\circle {0.2}}  \put(2,0.5){\makebox(0,0){$r^1_2$}}
   \put(4,0){\circle {0.2}}   \put(4,0.5){\makebox(0,0){$r^1_3$}}
   \put(6,0){\circle {0.2}}   \put(6,0.5){\makebox(0,0){$r^1_{k-1}$}}
   \put(8,0){\circle {0.2}}   \put(8,0.5){\makebox(0,0){$r^1_{k}$}}
 \put(0.1,0){\line(1,0){1.8}}
 \put(2.1,0){\line(1,0){1.8}}
\put(4.1,0){\line(1,0){0.5}}
\put(5.4,0){\line(1,0){0.5}}
\put(5,0){\makebox(0,0){$\cdots$}}
\put(6,0.1){\line(1,0){2}}
\put(6,-0.1){\line(1,0){2}}

\end{picture}
\caption{\label{p902}
${}^2A_{n-1}(q)$, where $n-1=2k-1$, $k=\left[\displaystyle\frac{n}{2}\right]$. }
\end{center}
\end{figure}
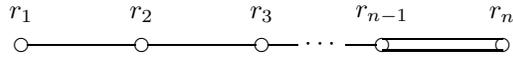
\begin{figure}
\begin{center}
\unitlength=8mm \begin{picture}(9,2)

  \put(0,0){\circle {0.2}}   \put(0,0.5){\makebox(0,0){$r_1$}}
   \put(2,0){\circle {0.2}}  \put(2,0.5){\makebox(0,0){$r_2$}}
   \put(4,0){\circle {0.2}}   \put(4,0.5){\makebox(0,0){$r_3$}}
   \put(6,0){\circle {0.2}}   \put(6,0.5){\makebox(0,0){$r_{n-1}$}}
   \put(8,0){\circle {0.2}}   \put(8,0.5){\makebox(0,0){$r_{n}$}}
 \put(0.1,0){\line(1,0){1.8}}
 \put(2.1,0){\line(1,0){1.8}}
\put(4.1,0){\line(1,0){0.5}}
\put(5.4,0){\line(1,0){0.5}}
\put(5,0){\makebox(0,0){$\cdots$}}
\put(6,0.1){\line(1,0){2}}
\put(6,-0.1){\line(1,0){2}}

\end{picture}
\caption{\label {p903}   
$B_{n}(q)$ and $C_{n}(q)$. }
\end{center}
\end{figure}

\begin{figure}
\begin{center}
\unitlength=8mm \begin{picture}(9,4)
   \put(0,2){\circle {0.2}}   \put(0,2.5){\makebox(0,0){$r_1$}}
   \put(2,2){\circle {0.2}}   \put(2,2.5){\makebox(0,0){$r_2$}}
   \put(4,2){\circle {0.2}}   \put(4,2.5){\makebox(0,0){$r_{l-3}$}}
   \put(6,2){\circle {0.2}}
 \put(5.8,2.5){\makebox(0,0){$r_{n-2}$}} \put(8,4){\circle
{0.2}} \put(8.63,4){\makebox(0,0){$r_{n-1}$}} \put(8,0){\circle
{0.2}} \put(8.5,0){\makebox(0,0){$r_n$}}
 \put(0.1,2){\line(1,0){1.8}}
\put(2.1,2){\line(1,0){0.5}}
\put(3.4,2){\line(1,0){0.5}}
\put(3,2){\makebox(0,0){$\cdots$}}
\put(4.1,2){\line(1,0){1.8}}
\put(6.08,2.08){\line(1,1){1.83}}
\put(6.08,1.92){\line(1,-1){1.83}}
\end{picture}
\caption{\label {p91}  
$D_n(q)$.}
\end{center}
\end{figure}

\begin{figure} \begin{center} \unitlength=9mm
\begin{picture}(9,6) \put(0,4){\circle {0.2}}
   \put(0,4.5){\makebox(0,0){$r_1$}} \put(2,4){\circle {0.2}}
   \put(2,4.5){\makebox(0,0){$r_2$}} \put(4,4){\circle {0.2}}
 \put(4,4.5){\makebox(0,0){$r_{n-3}$}} \put(6,4){\circle {0.2}}
   \put(5.8,4.5){\makebox(0,0){$r_{n-2}$}} \put(8,6){\circle
   {0.2}}   \put(8.7,6){\makebox(0,0){$r_{n-1}$}}
   \put(8,2){\circle {0.2}}
   \put(8.5,2){\makebox(0,0){$r_n$}}
 \put(0.1,4){\line(1,0){1.8}}
\put(2.1,4){\line(1,0){0.5}}
\put(3.4,4){\line(1,0){0.5}}
\put(3,4){\makebox(0,0){$\cdots$}}
\put(4.1,4){\line(1,0){1.8}}
\put(6.08,4.08){\line(1,1){1.83}}
\put(6.08,3.92){\line(1,-1){1.83}}
\put(4,3){\vector(0,-1){1.9}}
   \put(0,0){\circle {0.2}}   \put(0,0.5){\makebox(0,0){$r_1^1$}}
   \put(2,0){\circle {0.2}}   \put(2,0.5){\makebox(0,0){$r_2^1$}}
   \put(4,0){\circle {0.2}}
\put(4,0.5){\makebox(0,0){$r_{n-3}^1$}} \put(6,0){\circle {0.2}}
 \put(5.8,0.5){\makebox(0,0){$r_{n-2}^1$}} \put(8,0){\circle
{0.2}} \put(8,0.5){\makebox(0,0){$r_{n-1}^1$}}
 \put(0.1,0){\line(1,0){1.8}}
\put(2.1,0){\line(1,0){0.5}}
\put(3.4,0){\line(1,0){0.5}}
\put(3,0){\makebox(0,0){$\cdots$}}
\put(4.1,0){\line(1,0){1.8}}
\put(6,0.1){\line(1,0){2}}
\put(6,-0.1){\line(1,0){2}}
\end{picture}
\caption{\label {p92}
${^2D_n(q)}$.}
\end{center}
\end{figure}

\medskip

C~a~s~e~~(a). We have  $2\max\{m,n-m\}\geq m+(n-m)=n$. So by Lemma~\ref{r_not divi} the inequality
$$
r-2\geq\max\{m,n-m\}\geq n/2,
$$
holds, whence the statement~(1) of the lemma.

\medskip

C~a~s~e~~(b). Notice that if the order of  $L_m(q^2)$ is not divisible by $r$, then by Lemma~\ref{r_not divi} we obtain $r-2>r-3\geq 2m$.
Whence applying Lemma~\ref{r_not divi} to $U_{n-2m}(q)$ we obtain the inequality
$$
r-2\geq\max\{2m,n-2m\}\geq n/2,
$$ and as in Case (a), substituting  $m$ by  $2m$, we obtain item~(2) of the lemma.

\medskip

C~a~s~e~s~~(c,d). Since the orders of   $S_{2k}(q)$ and $O_{2k+1}(q)$ are equal, it is enough to consider case  (d). The inequality  $$3\max\{m,2(n-m)+1\}\geq 2m+2(n-m)+1=2n+1$$ and Lemma~\ref{r_not divi} imply
$$
r\geq \max\{m+2,2(n-m)+3\}=\max\{m,2(n-m)+1\}+2\geq \frac{2n+1}{3}+2=\frac{2n+7}{3},
$$
whence item~(3) (or (4) in case (c)) of the lemma holds.

\medskip

C~a~s~e~s~~(e,f). Since  $3\max\{m+1,2(n-m)\}\geq 2(m+1)+2(n-m)=2n+2$, by Lemma~\ref{r_not divi} the inequality
$$
r\geq \max\{m+2,2(n-m)+1\}=\max\{m+1,2(n-m)\}+1\geq \frac{2n+2}{3}+1=\frac{2n+5}{3}
$$
holds. Whence items~(5) and (6)  of the lemma hold.
\end{proof}

\begin{lemma} \label{non_non-degenerated_divisor}
Let  $L$ be isomorphic to $U_n(q)$ and $n\geq 12$, or to  $S_{2n}(q)$, $O_{2n+1}(q)$, $O_{2n}^+(q)$, or $O_{2n}^-(q)$ and $n\geq 6$. Assume an odd prime  $r$ does not divide the order of the stabilizer in  $L$ of a nondegenerate subspace  $U$ of the underlying space  $V$. Then one of the following statements hold:
\begin{itemize}
 \item[$(1)$]  $L=U_n(q)$ and $r\geq \left[\displaystyle\frac{n+1}{2}\right]+2\geq\displaystyle\frac{n+4}{2}$;
 \item[$(2)$]  $L=S_{2n}(q)$ and $r\geq 2\left[\displaystyle\frac{n+1}{2}\right]+3\geq  n+3$;
 \item[$(3)$]  $L=O_{2n+1}(q)$ and $r\geq  2\left[\displaystyle\frac{n+1}{2}\right]+1\geq n+1$;
 \item[$(4)$]  $L=O^+_{2n}(q)$ or  $L=O^-_{2n}(q)$  and $r\geq  2\left[\displaystyle\frac{n+1}{2}\right]+1\geq n+1$.
 \end{itemize}
\end{lemma}

\begin{proof}
Since an element of $L$ stabilizing a nondegenerate subspace  $U$ also stabilizes the (nondegenerate) orthogonal complement  $U^\perp$,
replacing, if necessary, $U$ with $U^\perp$,
we can assume  that for some  $m<n$ one of the following cases appear:
 \begin{itemize}
   \item[(a)] $L=U_n(q)$, $\dim U=m$ and $r$ does not divide the orders of  $U_m(q)$ and $U_{n-m}(q)$, where $1\leq m \leq n-1$;
   \item[(b)] $L=S_{2n}(q)$, $\dim U=2m$ and $r$ does not divide the orders of $S_{2m}(q)$ and $S_{2(n-m)}(q)$, where $1\leq m \leq n-1$;
   \item[(c)] $L=O_{2n+1}(q)$, $\dim U=2m$ and  $r$ does not divide the orders  of  $O^\pm_{2m}(q)$ and $O_{2(n-m)+1}(q)$, where $1\leq m \leq n-1$;
   \item[(d)] $L=O^\pm_{2n}(q)$, $\dim U=2m+1$ and  $r$ does not divide the orders of  $O_{2m+1}(q)$ and $O_{2(n-m-1)+1}(q)$, where  $1\leq m \leq n-1$;
   \item[(e)] $L=O^\varepsilon_{2n}(q)$,  $\varepsilon\in\{+,-\}$, $\dim U=2m$, ${\rm sgn}\,U=\delta\in\{+,-\}$  and $r$ does not divide the orders of  $O^\delta_{2m}(q)$ and $O^{\varepsilon\delta}_{2(n-m)}(q)$, where $1\leq m \leq n-2$.
 \end{itemize}

Consider all these cases separately.

\medskip

C~a~s~e~~(a). Since by Lemma~\ref{r_not divi} $$r-2\geq \max\{m,n-m\},$$ as in case  (a) in the proof of Lemma~\ref{non_parabolic_divisor}, we obtain item~(1) of the lemma.

\medskip

C~a~s~e~~(b). By Lemma~\ref{r_not divi}
$$r-3\geq \max\{2m,2(n-m)\}= 2\max\{m,n-m\}\geq 2\left[\displaystyle\frac{n+1}{2}\right]\geq n,$$
and item~(2) of the lemma follows.

\medskip

C~a~s~e~~(c). By Lemma~\ref{r_not divi}
$$r\geq \max\{2m+1,2(n-m-1)+3\}= 2\max\{m,n-m\}+1\geq 2\left[\displaystyle\frac{n+1}{2}\right]+1\geq n+1,$$
and item~(3) of the lemma holds.

\medskip

C~a~s~e~~(d). By Lemma~\ref{r_not divi}
$$r\geq \max\{2m+3,2(n-m-1)+3\}= 2\max\{m,(n-1)-m\}+3\geq 2\left[\displaystyle\frac{n}{2}\right]+3\geq n-1+3=n+2.$$
Since $n+2>n+1$, we obtain item~(4) of the lemma.

\medskip

C~a~s~e~~(e). By Lemma~\ref{r_not divi}
$$r\geq \max\{2m+1,2(n-m)+1\}= 2\max\{m,n-m\}+1\geq 2\left[\displaystyle\frac{n+1}{2}\right]+1\geq n+1,$$
and item~(4) of the Lemma holds.
\end{proof}

\section{Proof of Proposition~\ref{ex1}}

Let  $T$ be a set of transposition in $S_r$. Consider a graph  $\Gamma$ with the vertex set $\Omega=\{1,\dots,r\}$, where two different vertices  $i,j$ are adjacent if and only if $(ij)\in T$. If  $\Delta_1,\dots, \Delta_k$ are the connected components of $\Gamma$, then clearly  $$\langle T\rangle\leq\Sym\Delta_1\times\dots\times \Sym\Delta_k.$$ Now, if   $m=|T|<r-1$, then $\Gamma$ is disconnected and each $\Sym\Delta_i$ is a $\pi$-group (recall that $r=\min \pi'$). Therefore  $\langle T\rangle$ is a  $\pi$-group.

So for every proper subset $\pi$ of the set of all primes and for $r=\min \pi'$ every  $m<r-1$ transpositions generate a $\pi$-subgroup in $S_r$, while~$\Oo_\pi(S_r)=1$.
\qed

\section{Proof of Proposition~\ref{beta_A_n_prop}}

Notice that  $(12\dots r)=(12)(13)\dots(1r)$. Whence if  $x$ is a transposition, then  $\beta_{r,L}(x)\leq r-1.$ On the other hand, Proposition~\ref{ex1} shows that for a transposition $x$ the inequality $\beta_{r}(x,L)\geq r-1.$ holds. So item (1) of the proposition holds.

Now we consider $A_6$ and an involution $x$ lying outside of $S_6$. It is known $L=A_6\cong L_2(9)$ and
$x$ is a diagonal automorphism of $L_2(9)$. By Lemma \ref{alpha_classic} for $r\in \{3,5\}$ we have
$$\beta_{r}(x,L)\leq\alpha(x,L)=3.$$ So for  $r=5$
$$\beta_{r}(x,L)\leq 3<4=r-1,$$  therefore item (3) of the proposition for $A_6$ and $x$ holds. In order to prove item (2) of the proposition we need to show that $\beta_3(x,L)=3$. Since $\beta_3(x,L)\leq \alpha(x,L)=3$, we remain to show that $\beta_{3}(x,L)\ne 2$.

Assume that  $y\in x^L$ is  such that the order of $D=\langle x,y\rangle$ is divisible by 3. Since $D$ is a dihedral group, this means that $x$ inverts an element of order $3$ in  $\langle xy\rangle\leq L$. The group $L=A_6$ is known to contain exactly two classes of conjugate elements of order  3 with representatives $(123)$ and $(123)(456)$. It follows that every element of order $3$ of  $L$ is conjugate to its inverse. Thus  $x$ leaves invariant some (and hence any) conjugacy class of elements of order $3$. A contradiction, since by \cite[Exercise~2.18]{Wilson},  $x$ interchanges both classes.

It remains to show that if  $L=A_n$ and $x\in S_n$, then $\beta_{r}(x,L)\leq r-1$.
We use the induction by  $n$. By item (1) of the proposition we may assume that  $x$ is not a transposition.

Assume first that $n=5$. Then $r\in \{3,5\}$. If $x\in L$ is of odd order, then by Lemma~\ref{alpha_A_n}, $$\beta_{r}(x,L)\leq\alpha(x,L)=2=r-1.$$

So assume that  $x$ is an involution, i.e.  $x$ is a product of two independent transpositions. Since $$[(12)(45)][(13)(45)]=(123)\,\,\,\text{ and }\,\,\,[(13)(45)][(14)(23)]=(12345),$$ we conclude
$$\beta_{r}(x,L)=2\leq r-1.$$

Assume   $n=6$. Notice that  $L\cong L_2(9)$. Again $r\in \{3,5\}$. We may assume   $|x|\ne r$. In view of item (1)  we may also assume that $x$ is not a transposition and is not a product of three transpositions, since they are conjugate under outer automorphism of $A_6$.

If  $r=5$, then by Lemma \ref{alpha_A_n} we obtain $$\beta_{r}(x,L)\leq\alpha(x,L)=3\leq 4=r-1.$$
So assume   $r=3$. If  $|x|=5$, then Lemma  \ref{alpha_classic} implies $$\beta_{r}(x,L)\leq\alpha(x,L)=2=r-1.$$
Finally assume  $|x|=2$, i.e. $x$ is a product of two independent transpositions. As in the case  $n=5$, we derive~$\beta_{r}(x,L)=2= r-1.$

Let $n>6$. Since  $r\in \pi(L)$, the inequality  $r\leq n$ holds.

Assume  $x\in S_n$ has a fixed point. Then  $x$ is included in a point stabilizer isomorphic to  $S_{n-1}$. If $r\in \pi(S_{n-1})$, then $\beta_{r}(x,L)\leq r-1$
by induction. If $r\notin \pi(S_{n-1})$, then $r=n$ and by Lemma~\ref{alpha_A_n} we have $$\beta_{r}(x,L)\leq \alpha(x,L)\leq n/2=r/2<r-1.$$

Assume  $\langle x\rangle$ acts without fixed points, but not transitively. Then $x$ is contained in a subgroup of the form
$S_m\times S_{n-m}$, $2\leq m\leq n-2$ and the projections of $x$ on both components are nontrivial.  If the orders of the components are not divisible by  $r$, we obtain
$$
\beta_{r}(x,L)\leq \alpha(x,L)\leq n/2\leq\max\{m, n-m\}<r,
$$
whence the proposition follows. Assume  $r$ divides the order of at least one of the components in the product $S_m\times S_{n-m}$. If  $\max\{m,n-m\}\geq 5$, the induction implies the desired inequality. So we may assume  $\max\{m,n-m\}<5$ and $r=3$. Since $n>6$, we also obtain $\max\{m,n-m\}\geq n/2>3$ and so $\max\{m,n-m\}=4$. Thus $n\in\{7,8\}$ and   $S_m\times S_{n-m}$ is equal to either $S_3\times S_4$, or $S_4\times S_4$.
If $x\leadsto S_3$ then $\beta_{3}(x,L)\leq 2=r-1$. Since   $S_3$ is a homomorphic image of $S_4$, in order to finish the consideration of nontransitive action of $\langle x\rangle$ it remains to consider the following configuration: $n=8$ and $x$ is a product of four independent transpositions. In this case  $x\in S_2\times S_6$ and   $x\leadsto S_6$, so by induction
$\beta_{3}(x,L)\leq 2=r-1$.

Finally assume that  $\langle x\rangle$ acts transitively. Since  $x$ is of prime order, this implies $n=|x|$ is a prime, without loss of generality we may assume
$$
x=(12\dots n).
$$
Let $y=(123)$. Then
$$
x^{-1}(x^{y^{-1}})=x^{-1}yxy^{-1}=(134).
$$
Thus the subgroup $\langle x, x^{y^{-1}}\rangle$ contains a $3$-cycle and is primitive, since it is transitive and $n$ is prime. By the Jordan theorem   \cite[Theorem 3.3E]{DM}, we conclude $L\leq \langle x, x^{y^{-1}}\rangle$, whence $$\beta_{r}(x,L)\leq \alpha(x,L)=2\leq r-1,$$
and the proposition follows. \qed

\section{Proof of Theorem~\ref{t4}}
Theorem~\ref{t4} is true, if $\alpha(x,L)\leq 11$. So by Lemmas,
\ref{Sporadic}, \ref{Exceptional}, and \ref{Small_rank}, we derive that Theorem~\ref{t4} holds for the  sporadic groups, for the exceptional groups of Lie type, for $A_{n-1}(q)=L_n(q)$ and ${}^2A_n(q)=U_n(q)$ if $n\leq 11$, for $B_n(q)=O_{2n+1}(q)$  and $C_{n}(q)=S_{2n}(q)$ if $n\leq 5$, and for $D_n(q)=O_{2n}^+(q)$ and ${}^2D_n(q)=O_{2n}^-(q)$ if $n\leq 5$.

The theorem also holds in the case when  $L=A_n$ is an alternating group. Indeed, if  $|L|=n!$ is not divisible by  $r$, then $r>n$ and by Lemma  \ref{alpha_A_n}$$\alpha(x,L)\leq n-1<r-1,$$ i.e. item  (2) of Theorem~\ref{t4} holds. If  $r$ divides $|L|$ and  $n\ne 6$, then by Proposition~\ref{beta_A_n_prop}, we have
   $$
   \beta_r(x,L)\leq r-1\leq 2(r-2).
   $$
For  $n=6$ the inequality $\alpha(x,L)\leq 5<11$ holds and the theorem follows.

Thus by the classification of finite simple groups   \cite[Theorem~0.1.1]{AschLyoSmSol} it remains to prove that Theorem~\ref{t4} holds in the following cases:
\begin{itemize}
  \item  $L=A_{n-1}(q)=L_n(q)$ or $L={}^2A_{n-1}(q)=U_n(q)$  for $n\geq12$;
  \item  $L=B_{n}(q)=O_{2n+1}(q)$ or $L=C_{n}(q)=S_{2n}(q)$ for $n\geq6$;
  \item  $L=D_{n}(q)=O^+_{2n}(q)$ or $L={}^2D_{n}(q)=O^-_{2n}(q)$ for $n\geq6$.
 \end{itemize}
Assume that Theorem~\ref{t4} is not true and choose a nonabelian simple group $L$ of minimal order possessing an automorphism  $x$ of prime order such that the theorem is not true. Thus  $L$ is one of the classical groups mentioned above and, for some prime $r$, one of the following conditions holds
\begin{itemize}
  \item[(a)] $\alpha(x,L)\geq 12$,
  \item[(b)] if   $r$ does not divides  $|L|$, then $\alpha(x,L)>r-1$,
  \item[(c)] if $r$ divides  $|L|$, then $\beta_r(x,L)>2(r-2)$.
\end{itemize}

We prove a series of steps that lead to the final contradiction.

\medskip
\begin{itemize}
  \item[$(i)$] {\em  $r$ divides $|L|$ and $\beta_r(x,L)>2(r-2)$.}
\end{itemize}
\medskip

Assume  $r$ does not divides  $|L|$. Then by Lemma~\ref{r_not divi} one of the following statements holds:
 \begin{itemize}
  \item  $L=L_n(q)$ or $L=U_n(q)$ and $n\leq r-2$, whence by Lemma~\ref{alpha_classic} we obtain $\alpha(x,L)\leq n\leq r-2$, a contradiction with~(b);
  \item  $L=O_{2n+1}(q)$ or $L=S_{2n}(q)$ and  $2n\leq r-3$,  by Lemma~\ref{alpha_classic} we obtain $\alpha(x,L)\leq 2n+1\leq r-2$, a contradiction with~(b);
  \item  $L=O^+_{2n}(q)$ or $L=O^-_{2n}(q)$ and $2n\leq r-1$, again by Lemma \ref{alpha_classic} we obtain  $\alpha(x,L)\leq 2n\leq r-1$; a contradiction with~(b).
 \end{itemize}

Hence  $r$ divides  $|L|$ and   $\beta_r(x,L)>2(r-2)$ in view of  (c).  Thus we obtain  $(i)$.

\medskip

Now using the bounds on  $\alpha(x,L)$ from Lemma~\ref{alpha_classic} and the fact that  $L$ possesses an automorphism  $x$ such that $\alpha(x,L)\geq\beta_r(x,L)\geq 2r-3$ we conclude that

\medskip
\begin{itemize}
  \item[$(ii)$] {\em  The following statements hold.
\begin{itemize}
  \item[$(1)$] If $L=L_n(q)$ or  $L=U_n(q)$, then $n\geq 2r-3$;
  \item[$(2)$] If  $L$ is one of the group   $S_{2n}(q)$, $O_{2n+1}(q)$  or $O^\pm_{2n}(q)$, then $n\geq 2r-6$, possibly excepting the following cases\footnote{In view of the isomorphism $S_{2n}(q)\cong O_{2n+1}(q)$ for  $q$ even, we need not to consider the case $L=O_{2n+1}(q)$ and $q$ is even separately.}:
      \begin{itemize}
        \item[$(2a)$] $L=S_{2n}(q)$,  $q$ is odd, $x$ is an inner automorphism induced by a transvection; in this case  $n\geq r-1$;
        \item[$(2b)$] $L=S_{2n}(q)$,  $q$ is even, $x$ is an inner automorphism induced by a transvection; in this case $n\geq r-2$;
        \item[$(2c)$] $L=O_{2n+1}(q)$,  $q$ is odd, $x$ is an inner-diagonal automorphism, induced by a reflection, in this case   $n\geq r-2$;
        \item[$(2d)$] $L=O^\pm_{2n}(q)$,  $q$ is odd, $x$ is an inner-diagonal automorphism, induced by a reflection; in this case $n\geq r-2$;
        \item[$(2e)$] $L=O^\pm_{2n}(q)$,  $q$ is even, $x$ is an inner automorphism induced by a transvection; in this case   $n\geq r-2$.
      \end{itemize}
   \end{itemize}
  }
\end{itemize}
\medskip

\begin{itemize}
  \item[$(iii)$] {\em Cases $(2a)$--$(2e)$ in step  $(ii)$ are impossible. If $L$ is one of the groups  $S_{2n}(q)$, $O_{2n+1}(q)$  or $O^\pm_{2n}(q)$, then $\alpha(x,L)\leq n+3$.}
\end{itemize}
\medskip

Assume, one of cases  $(2a)$--$(2e)$ holds. The inequality   $n\geq r-2$ is satisfied in all these cases, and it implies for $r\geq 7$
$$2(n-1)\geq 2(r-3)> r-1>r-3.$$ By Lemma~\ref{r_not divi}  the orders of $S_{2(n-1)}(q)$, $O_{2n-1}(q)$, and $O^\pm_{2(n-1)}(q)$ are divisible by~$r$. Clearly, the orders of $S_{2(n-1)}(q)$, $O_{2n-1}(q)$, and $O^\pm_{2(n-1)}(q)$ are also divisible by~$r$ for $r=3,5$.

We claim that any transvection of $S_{2n}(q)$ and $O^\pm_{2n}(q)$ (in the last case  $q$ is even) is contained as a central element in a subgroup  $H$ such that $H/Z(H)\cong S_{2(n-1)}(q)$ and $H/Z(H)\cong O^\pm_{2(n-1)}(q)$ respectively. Indeed, all transvection in these groups are conjugate. Consider the stabilizer of the decomposition of  $V$ into an orthogonal sum of nondegenerate  subspaces $U$ and $W$ of dimensions   $2$ and $2(n-1)$ respectively, and let $H$ be a subgroup in this stabilizer consisting of all elements acting on $U$ as scalars. Clearly,  $H$ contains a transvection and has the desired structure.

Now notice that for every of such subgroup  $H$ the order  $|H/Z(H)|$ is divisible by  $r$. Since  $L$ is a minimal counter example, it follows  that in cases  $(2a)$, $(2b)$, and $(2e)$
$$
\beta_r(x,L)\leq 2(r-2).
$$

Since the reflections in orthogonal groups over fields of odd characteristic are conjugate, every reflection induces a reflection on a nondegenerate subspace of codimension $2$ and acts identically on its orthogonal complement. Therefore, every reflection in   $O_{2n+1}(q)$  or $O^\pm_{2n}(q)$ with  $q$ odd, normalizes but not centralizes a subgroup  $H$ such that   $H/Z(H)\cong O_{2n-1}(q)$ or $H/Z(H)\cong O^\pm_{2(n-1)}(q)$. Again the order  $|H/Z(H)|$ is divisible by  $r$, and the minimality of  $L$ implies that in cases $(2c)$ and $(2d)$
$$
\beta_r(x,L)\leq 2(r-2).
$$

Now if  $L$ is one of $S_{2n}(q)$, $O_{2n+1}(q)$, or $O^\pm_{2n}(q)$, then $\alpha(x,L)\leq n+3$ by Lemma~\ref{alpha_classic}.
\medskip

\begin{itemize}
  \item[$(iv)$] {\em  $x$ is not   unipotent.
  }
\end{itemize}
\medskip

Otherwise by Lemma~\ref{Parabolic} we may assume that  there exists a maximal parabolic subgroup  $P$ such that   $x$ is contained in  $P\setminus U$, where  $U$ is the unipotent radical of  $P$. Therefore  $x$ has a nontrivial image in~$P/\Oo_\infty(P)$ and one of the following cases holds:
\begin{itemize}
  \item[$(iv1)$] $L=L_n(q)$; $P$ is corresponding to a set  $J$ of fundamental roots of the Dynkin diagram pic.~\ref{p90}, where either $J=\{r_2,\dots,r_{n-1}\}$ or $J=\{r_1,\dots,r_{n-2}\}$;
      $P/\Oo_\infty(P)$ is isomorphic to a subgroup of $\Aut(L_{n-1}(q))$, containing~$L_{n-1}(q)$.
  \item[$(iv2)$] $L=U_n(q)$; $P$ is corresponding to a set $J^1$ of fundamental roots of the Dynkin diagram pic.~\ref{p901} and~\ref{p902}, where either $J^1=\{r_2^1,\dots, r_{[n/2]}^1\}$ and $P/\Oo_\infty(P)$ is isomorphic to a subgroup of $\Aut(U_{n-2}(q))$ containing  $U_{n-2}(q)$, or $J^1=\{r_1^1,\dots, r_{[n/2]-1}^1\}$ and $P/\Oo_\infty(P)$ is isomorphic to a subgroup of $\Aut(L_{[n/2]}(q^2))$ containing $L_{[n/2]}(q^2)$.
  \item[$(iv3)$] $L=S_{2n}(q)$ or $L=O_{2n+1}(q)$; $P$ is corresponding to a set $J$ of fundamental roots of the Dynkin diagram pic.~\ref{p903}, where either $J=\{r_2,\dots, r_{n}\}$ and  $P/\Oo_\infty(P)$ is isomorphic to a subgroup of  $\Aut(S_{2(n-1)}(q))$, containing $S_{2(n-1)}(q)$ or to a subgroup of  $\Aut(O_{2n-1}(q))$ containing $O_{2n-1}(q)$, or  $J=\{r_1,\dots, r_{n-1}\}$ and  $P/\Oo_\infty(P)$ is isomorphic to a subgroup of   $\Aut(L_{n}(q))$ containing~$L_{n}(q)$.
  \item[$(iv4)$] $L=O^+_{2n}(q)$;
  $P$ is corresponding to a set $J$ of fundamental roots in the Dynkin diagram pic.~\ref{p91}, where either  $J=\{r_2,\dots, r_{n}\}$ and  $P/\Oo_\infty(P)$ is isomorphic to a subgroup of $\Aut(O^+_{2(n-1)}(q))$ containing  $O^+_{2(n-1)}(q)$, or  $J=\{r_1,\dots, r_{n-1}\}$ and  $P/\Oo_\infty(P)$ is isomorphic to a subgroup of $\Aut(L_{n}(q))$ containing  $L_{n}(q)$.
  \item[$(iv5)$] $L=O^-_{2n}(q)$;
  $P$ is corresponding to a set  $J^1$ of fundamental roots in the Dynkin diagram pic.~\ref{p92}, where either   $J^1=\{r^1_2,\dots, r^1_{n-1}\}$ and  $P/\Oo_\infty(P)$ is isomorphic to a subgroup of  $\Aut(O^-_{2(n-1)}(q))$  containing $O^-_{2(n-1)}(q)$, or  $J^1=\{r^1_1,\dots, r^1_{n-2}\}$ and  $P/\Oo_\infty(P)$ is isomorphic to a subgroup of  $\Aut(L_{n-1}(q))$ containing $L_{n-1}(q)$.
  \end{itemize}

We consider all possibilities case by case and show that  $\beta_r(x,L)\leq \max\{11, 2(r-2)\}$, thus obtaining a contradiction.

In view of step  $(ii)$, for both $(iv1)$ and $(iv2)$ the inequality  $n\geq 2r-3$ holds. Moreover $n\geq12$. Whence  $n-1,n-2$, and $2[n/2]$ are greater than $r-2$, and by Lemma~\ref{r_not divi} the orders of  $L_{n-1}(q)$, $U_{n-2}(q)$, and  $L_{[n/2]}(q^2)$  are divisible by  $r$. By the minimality of  $L$ we derive
$$\beta_r(x,L)\leq \max\{11, 2(r-2)\}.$$

Also by step $(ii)$ in case  $(iv3)$ the inequality  $n\geq 2r-6$ holds. Moreover  $n\geq6$. Whence  $n> r-2$ and $2(n-1)>r-3$, and by Lemma~\ref{r_not divi} the orders of  $S_{2(n-1)}(q)$, $O_{2n-1}(q)$, and  $L_{n}(q)$ are divisible by  $r$. Again by the minimality of  $L$ we obtain
$$\beta_r(x,L)\leq \max\{11, 2(r-2)\}.$$

Finally, by step $(ii)$ in cases  $(iv4)$ and $(iv5)$ the inequality  $n\geq2r-6$ holds, and also  $n\geq6$. Whence  $n>n-1> r-2$ and  $2(n-1)>r-1$, and by Lemma~\ref{r_not divi} the orders of  $O^\pm_{2(n-1)}(q)$, $L_{n}(q)$, and $L_{n-1}(q)$ are divisible by $r$. Like above, by the minimality of $L$ we conclude that $$\beta_r(x,L)\leq \max\{11, 2(r-2)\}.$$

\begin{itemize}
  \item[$(v)$] {\em The automorphism  $x$ is not induced by an irreducible semisimple inner-diagonal element. }
\end{itemize}
\medskip

This step follows immediately by Lemma~\ref{alpha_irreducible}.

\begin{itemize}
  \item[$(vi)$] {\em The automorphism  $x$ is not induced by a semisimple inner-diagonal element contained in a proper parabolic subgroup of~$\L$.  }
\end{itemize}

\medskip

Assume by contradiction that   $x$ is induced by a semisimple inner-diagonal element lying in a proper parabolic subgroup. By Lemma ~\ref{Parabolic},  $x$ is conjugate with an element from the Levi complement   $J$ of a maximal parabolic subgroup  $P$, and normalizes but not centralizes every component of this complement. Then  $x$ induces a nontrivial automorphism on each component, and, if the order of a component is divisible by  $r$, the minimality of  $L$ implies that
$$\beta_r(x,L)\leq \max\{11, 2(r-2)\}.$$
If $r$ does not divide the order of all components, then $r$ does not divide the order of $P$. By Lemma~\ref{non_parabolic_divisor} one of the following possibilities occur:
\begin{itemize}
 \item
  $L=L_n(q)$  or  $L=U_n(q)$ and $r\geq \displaystyle\frac{n+4}{2}$; a contradiction with the inequality   $n\geq 2r-3$ obtained in step~$(ii)$.
  \item
  $L=S_{2n}(q)$ or  $L=O_{2n+1}(q)$ and $r\geq  \displaystyle\frac{2n+7}{3}$; a contradiction with $n\geq 2r-6$ obtained in step~$(ii)$, since $n\geq 6$.
 \item
  $L=O^\pm_{2n}(q)$ and $r\geq  \displaystyle\frac{2n+5}{3}$. If $(n,r)\ne(8,7)$, then we obtain a contradiction with $n\geq 2r-6$ obtained in step~$(ii)$, since  $n\geq 6$. If $(n,r)=(8,7)$, then by Lemma~\ref{alpha_classic} we have
     $$
     \beta_r(x,L)\leq \alpha(x,L)\leq n+3=11\leq\max\{11, 2(r-2)\}.
     $$

 \end{itemize}

\begin{itemize}
  \item[$(vii)$] {\em If $L$ is one of the group  $U_n(q)$, $S_{2n}(q)$, $O_{2n+1}(q)$, or  $O_{2n}^\pm(q)$, then the automorphism   $x$ is not induced by a similarity of the underlying space~$V$ of~$L$, possessing a proper nondegenerate invariant subspace.  }
\end{itemize}

Assume  $U$ is a nontrivial nondegenerate $x$-invariant subspace of minimal possible dimension. Set $W=U^\perp$. Then  $W$ is also $x$-invariant, and we have $V=U\oplus W$ and $\dim W\geq \displaystyle\frac{1}{2}\dim V$ since $U$ is nondegenerate. If $\dim U=1$, we may assume that a preimage   $x^*$ of  $x$ in the group of all similarities of $V$ is nonscalar on $W$
(see Lemma~\ref{SemisimpleInvariant}). If $\dim U>1$ , then $x$ is nonscalar on $W$ by the choice of $U$ (otherwise $U$ would not be of minimal possible dimension). Let  $H$ be the stabilizer of both $U$ and $W$ in  $\L$. Then   $x\in H$. Consider the projective image $P\Delta(W)$ of the group of all similarities $\Delta(W)$ of~$W$. Clearly  $P\Delta(W)$ is a homomorphic image of $H$ and the image  $\overline{x}$ of $x$ in $P\Delta(W)$ is nontrivial. Let  $S$ be the socle of  $P\Delta(W)$. If $r$ divides $|S|$, then by the choice of  $L$ we have
$$
\beta_r(x,L)\leq \beta_{r}(\overline{x},S)\leq \max\{11, 2(r-2)\}.
$$
Similarly, if   $\dim U>1$, and  the order of the socle  $T$ of   $P\Delta(U)$ (here $P\Delta(U)$ is the projective image of the group of similarities $\Delta(U)$ of $U$), is divisible by   $r$ we derive
$$
\beta_r(x,L)\leq \max\{11, 2(r-2)\}.
$$
So we may assume that  $r$ does not divide the order of the stabilizer in  $\L$ of some decomposition  $V=U\oplus W$ into the sum of mutually orthogonal nondegenerate subspaces $U$ and~$W$.

By Lemma~\ref{non_non-degenerated_divisor} one of the following cases holds:
\begin{itemize}

 \item  $L=U_n(q)$ and $r\geq
 \displaystyle\frac{n+4}{2}$; a contradiction with the inequality  $n\geq 2r-3$ obtained in Step~$(ii)$.
 \item  $L=S_{2n}(q)$ and $r\geq    n+3$;  a contradiction with the inequality $n\geq 2r-6$ obtained in Step~$(ii)$ and condition $n\geq 5$.
 \item  $L=O_{2n+1}(q)$ and $r\geq n+1$;  a contradiction with the inequality $n\geq 2r-6$ obtained in Step~$(ii)$ and condition  $n\geq 6$.
 \item  $L=O^+_{2n}(q)$ or  $L=O^-_{2n}(q)$  and $r\geq  n+1$;  a contradiction with the inequality $n\geq 2r-6$ obtained in Step~$(ii)$  and condition $n\geq 6$.
 \end{itemize}

 \begin{itemize}
  \item[$(viii)$] {\em   $x$ is not inner-diagonal.  }
\end{itemize}

Follows by $(iv)$--$(vii)$

 \begin{itemize}
  \item[$(viii)$] {\em    $x$ is not a field automorphism modulo  $\L$.
  }
\end{itemize}
Otherwise by Lemma~\ref{Field_Aut} we may assume that   $x$ is a canonical field automorphism. Then   $x$ induces a nontrivial field automorphism on the stabilizer  $H$ of a $2$-dimensional nondegenerate ($1$-dimensional in case $L=L_n(q)$) subspace (such that the restriction of the corresponding quadratic form on the subspace have the sign  $\varepsilon$,  if $L=O_{2n}^\varepsilon(q)$), and on the socle $S$ of  $H/\Oo_\infty(H)$. At that
\begin{itemize}
  \item[$(viii1)$] $S\cong L_{n-1}(q)$, if $L=L_n(q)$;
  \item[$(viii2)$] $S\cong U_{n-2}(q)$, if $L=U_n(q)$;
  \item[$(viii3)$] $S\cong S_{2(n-1)}(q)$, if $L=S_{2n}(q)$;
  \item[$(viii4)$] $S\cong O_{2n-1}(q)$, if  $L=O_{2n+1}(q)$;
  \item[$(viii5)$] $S\cong O^+_{2(n-1)}(q)$, if $L=O^\varepsilon_{2n}(q)$.
  \end{itemize}
As above, if  $|S|$ is divisible by $r$, the induction implies  $$
\beta_r(x,L)\leq \max\{11, 2(r-2)\}.
$$
Otherwise  by Lemma~\ref{r_not divi} we obtain on of the following inequalities
\begin{itemize}
  \item[$(viii1)$] $r-2\geq n-1$, a contradiction with the inequality   $n\geq 2r-3$ obtained in Step~$(ii)$;
  \item[$(viii2)$] $r-2\geq n-2$,  a contradiction with the inequality  $n\geq 2r-3$, obtained in Step~$(ii)$ or with  $n\geq 12$;
  \item[$(viii3,4)$]  $r-3\geq 2(n-1)$, a contradiction with the inequality   $n\geq 2r-6$, obtained in Step~$(ii)$ or with  $n\geq 5$;
  \item[$(viii5)$] $r-1\geq 2(n-1)$, a contradiction with the inequality    $n\geq 2r-6$,  obtained in Step~$(ii)$, or with $n\geq 6$.
  \end{itemize}

Thus by Lemma~\ref{Aut} it follows that

\begin{itemize}
  \item[$(ix)$] {\em Modulo~$\L$, $x$ either is graph-field and $L\in\{L_n(q),O_{2n}^+(q)\}$, or is graph and \\ $L\in\{L_n(q),U_n(q),O_{2n}^+(q),O_{2n}^-(q)\}$. Moreover $|x|=2$.
  }
\end{itemize}

We exclude both remaining possibilities for $x$.
\begin{itemize}
  \item[$(x)$] {\em  $x$ is not a graph-field automorphism modulo~$\L$.
  }
\end{itemize}
Suppose the contrary. Then by Lemma~\ref{Field_Aut} we have  $q=q_0^2$ and $C_L(x)\cong U_{n}(q_0)$, if $L=L_n(q)$; and $C_L(x)\cong O_{2n}^-(q_0)$, if $L=O^+_{2n}(q)$. It is easy to see that the index  $|L:C_L(x)|$ is even. Moreover,  $Z(C_L(x))=1$. By Lemma~\ref{guest}, $x$ normalizes but not centralizes a subgroup $H=C_L(x)^g$. Therefore $x$ induces on $H$ a nontrivial automorphism  $\overline{x}$.
\begin{itemize}
  \item If $L=L_n(q)$, then $n\geq 2r-3>r-2$, $H\cong  U_{n}(q_0)$, and by Lemma~\ref{r_not divi},  $r$ divides~$|H|$.

  \item  If $L=O_{2n}^+(q)$, then $H\cong  O^-_{2n}(q_0)$, $n\geq 2r-6$, whence, using the inequality  $n\geq 6$, we conclude $2n>r-1$, and by Lemma~\ref{r_not divi},  $r$ divides~$|H|$.
 \end{itemize}
Therefore, by induction we obtain
      $$
      \beta_r(x,L)\leq \beta_{r}(\overline{x},H)\leq \max\{11, 2(r-2)\};
      $$
      a contradiction.

\begin{itemize}
  \item[$(xi)$] {\em  $x$ is not a graph automorphism modulo~$\L$. }
\end{itemize}

Suppose the contrary. Then one of the following possibilities occurs: $L=L^\varepsilon_n(q)$ or  $L=O^\varepsilon_{2n}(q)$, $\varepsilon\in\{+,-\}$. Consider these possibilities separately.

\begin{itemize}
  \item  $L=L^\varepsilon_n(q)$, $\varepsilon\in\{+,-\}$.
\end{itemize}

By Lemma~\ref{GraphAutGLU},  $x$ normalizes but not centralizes a subgroup  $H$ of $L$, isomorphic to $O_n(q)=O_{2k+1}(q)$, if $n=2k+1$, and a subgroup isomorphic to either $S_n(q)=S_{2k}(q)$, or $O^\pm_n(q)=O^\pm_{2k}(q)$, if $n=2k$. Then  $x$ induces on  $H$ an automorphism $\overline{x}$. Since  $n\geq 2r-3$ and $n\geq 12$, in all cases we have
       $$2k\geq n-1\geq 2r-4>r-1>r-3.
       $$
Whence  $|H|$ is divisible by $r$ by Lemma~\ref{r_not divi}, and the induction implies
      $$
      \beta_r(x,L)\leq \beta_{H}(\overline{x},H)\leq \max\{11, 2(r-2)\}.
      $$

\begin{itemize}
  \item  $L=O^\varepsilon_{2n}(q)$, $\varepsilon\in\{+,-\}$.
\end{itemize}

As above, first we choose a  subgroup  $H$ of $L$  such that $x$ normalizes but not centralizes it.

If  $q$ is even, then $L=\L$ and by Lemma \ref{Graph_Inv_D_n_even_q} a graph automorphism $\gamma$ of $\L$ normalizes but not centralizes a subgroup isomorphic to $ O_{2n-2}^{\varepsilon\eta}(q)$ for appropriate $\eta\in\{+,-\}$.

Assume that  $q$ is odd. In this case by Lemma~\ref{Graph_Inv_D_n},  $x$ normalizes but not centralizes a subgroup  $H$ of $L$, such that either  $H\cong O_{2n-1}(q)$, or $n=2k+1$ and $H\cong O_{2k+1}(q^2)$.

Since  $n\geq 2r-6$ and $n\geq 6$, we obtain that  $$2(n-1)\geq n>r-3.$$ Whence  $|H|$ is divisible by $r$ by Lemma~\ref{r_not divi}, and by induction we have $$
      \beta_r(x,L)\leq \beta_{r}(\overline{x},H)\leq \max\{11, 2(r-2)\}.
      $$

\section{Proofs of Theorems~\ref{t2} and~\ref{t3}}

Let  $\pi$ be a proper subset of the set of all primes, $r$ be the minimal prime not in $\pi$ and $$m:=\max\{11, 2(r-2)\}.$$  In view of Proposition~\ref{ex1}, in order to prove Theorems~\ref{t2} and~\ref{t3} it is enough to prove that  $\BS_\pi^m$ coincides with the class of all finite groups.

Assume the contrary, and let  $G\notin \BS_\pi^m$ be of minimal order. By Lemma~\ref{2notinpi} we obtain that  $r>2$.

By Lemma~\ref{red} we conclude that $G$ is isomorphic to an almost simple group with the simple socle  $L$, satisfying to the following properties:  $L$ is neither a $\pi$- nor a $\pi'$-group, it admits an automorphism  $x$ of prime order lying in  $\pi$ such that every $m$ conjugates of $x$ generate a  $\pi$-group, and $G\cong \langle L,x\rangle$.

Since   $L$ is not a $\pi$-group, there exists a prime divisor  $s$ of the order of~$L$, not lying in~$\pi$. Clearly $s\geq r$.

Now  $\alpha(x,L)> m$, since otherwise $m$ conjugates of $x$ generate a subgroup of $\langle L,x\rangle$ whose order is divisible by  $s$ in contrast with the assumption that every $m$ conjugates of $x$ generate a $\pi$-group. In particular,  $$\alpha(x,L)>11\text{ and } \alpha(x,L)>2(r-2)\geq r-1.$$ The same arguments imply   $$\text{ if }r\text{ divides the order of }L,\text{ then }\beta_r(x,L)>m\geq 2(r-2).$$ A contradiction with Theorem~\ref{t4}.

\bibliographystyle{amsplain}

\begin{thebibliography}{100}


\addcontentsline{toc}{chapter}{The bibliography}


\bibitem{AschLyoSmSol} M.~Aschbacher, R.~Lyons, S.D.~Smith, R.~Solomon, The classification of finite simple groups. Groups of characteristic 2 type. Mathematical Surveys and Monographs, 172. American Mathematical Society, Providence, RI, 2011. xii+347 pp.

\bibitem{AsSeitz} { M.Aschbacher, G.M.Seitz}, Involution in Chevalley groups over fields of even order, { Nagoya Math. J.}, vol. 63 (1976), 1--91.


\bibitem{Baer}  R. Baer, Engelsche Elemente Noetherscher Gruppen, Math. Ann.,
     133 (1957),  256–270.


\bibitem{Suz}  M. Suzuki, Finite groups in which the centralizer of any element of order 2 is 2-closed,  Ann.
of Math. (2), 82 (1968), 191–212


\bibitem{AlpLy}  J. Alperin, R. Lyons, On conjugacy classes of $p$-elements, J. Algebra,
     19 (1971), N2, 536–537.


\bibitem{Mamont} A. S. Mamontov, An Analog of the Baer-Suzuki Theorem for Infinite Groups, Sib Math. J.,  45 (2004), N2, 327--330.

\bibitem{OMS} V. D. Mazurov, A. Yu. Ol'shanskii, A. I. Sozutov, Infinite groups of finite period, Algebra and Logic, 54 (2015), N2, 161--166.


\bibitem{Soz} A. I. Sozutov, On a generalization of the Baer-Suzuki theorem, Sib. Math. J.,   45 (2000), N3, 561--562.




\bibitem{Gu} S. Guest, A solvable version of the Baer-Suzuki Theorem,  Trans. Amer. Math. Soc., 2010 , 362(11), p. 5909--5946.

    \bibitem{GuestLevy} { S. Guest, D. Levy}, Criteria for solvable radical membership via $p$-elements, { Journal of Algebra}, vol. 415 (2014), 88--111.


\bibitem{FGG} P. Flavell, S. Guest, R. Guralnick, Characterizations of the solvable
      radical, Proc. Amer. Math. Soc.,  138 (2010), N4, 1161--1170.


\bibitem{GGKP}  N. Gordeev, F. Grunewald, B. Kunyavskii, E. Plotkin, A description
        of Baer–Suzuki type of the solvable radical of a finite group, J. Pure
        Appl. Algebra, 213 (2009), N2, 250--258.


\bibitem{GGKP1}  N. Gordeev, F. Grunewald, B. Kunyavskii, E. Plotkin, Baer–Suzuki theorem for the solvable radical of a finite group,  C. R. Acad. Sci. Paris, Ser. I, 347 (2009), N5--6, 217--222.

\bibitem{GGKP2}  N. Gordeev, F. Grunewald, B. Kunyavskii, E. Plotkin, From Thompson to Baer-Suzuki: a sharp characterization of the solvable radical, J. Algebra 323(10), (2010), 2888--2904.



\bibitem{Car}
{R. W. Carter}
{Simple groups of Lie type},
John Wiley and Sons, London,~1972.



\bibitem{Atlas} J. H. Conway, R. T. Curtis, S. P. Norton, R. A. Parker, R. A. Wilson
Atlas of Finite Groups, Clarendon Press, Oxford, 1985.

\bibitem{Bray}
{J.~N.~Bray, D.~F.~Holt, C.~M.~Roney-Dougal}, The Maximal Subgroups of the Low-Dimensional Finite Classical Groups. Cambridge: Cambridge Univ. Press, 2013. 438~p.

\bibitem{KL}
{P. B.
Kleidman,  M.
Liebeck}
{The subgroups structure of finite
classical groups},
Cambridge Univ. Press, 1990.


\bibitem{SuzI} M.Suzuki, Group Theory I, NY, Springer-Verlag, 1982.

\bibitem{GLS} D.Gorenstein, R.Lyons, R.Solomon, The classification of the finite simple groups. Number 3. American Mathematical Society, Providence, RI, 1998.


\bibitem{Aschbacher} M.Aschbacher, Finite Group Theory, Cambrige Univ. Press, 1986.

\bibitem{Wilson} R.A.Wilson, The Finite Simple Groups, Springer-Verlag, London, 2009.

\bibitem{DM} J.D.Dixon, B.Mortimer, Permutation groups, Springer-Verlag, N. Y., 1996

\bibitem{BS_odd}  D. O. Revin, {On Baer-Suzuki $\pi$-theorems}, Sib Math. J., {52}:2  (2011),  340--347.

\bibitem{BS_Dpi}  D. O. Revin, On a relation between the Sylow and Baer-Suzuki theorems, Sib. Math. J., 52:5 (2011), 904--913.

\bibitem{Palchik} E. M. Pal'chik, {O porozhleniyah parami sopryazhennyh elementov v konechnyh gruppah}, Doklady NAN Belarusi, {55}  (2011), N4 19--20 (in Russian).

\bibitem{Tyut}
{V.~N. Tyutyanov,} On the existence of solvable normal subgroups in finite groups, Math. Notes, 61 (1997) N 5, 632--634.

\bibitem{Tyut1}  V. N. Tyutyanov, {Kriteriy neprostoty dlya konechnoy gruppy}, Izv. Gomel'skogo gos. univ. im F. Skoriny (Voprosy algebry), (2000),
N3 (16), 125--137 (in Russian).

\bibitem{LieSaxl} { M.W.Liebeck, J.Saxl}, Minimal degrees of primitive permutation group, with an application to monodromy groups of covers of Riemann surfaces, { Proceedings of London Mathematical Society}, vol. 63 (1991), 266--314.


\bibitem{GS} R. Guralnick, J. Saxl, Generation of finite almost simple groups by conjugates, J. Algebra, 268 (2003), N2, 519--571.


\bibitem{DiMPZ} L. Di Martino, M. A. Pellegrini, A. E. Zalesski, On Generators and Representations of the Sporadic Simple Groups, Comm. Algebra, 42 (2014), N2, 880--908.

\end{thebibliography}

\bigskip


\end{document}